\documentclass[12pt]{article}
\usepackage{a4}
\usepackage[british]{babel}

\usepackage[font=small]{caption}
\usepackage{amsthm}
\usepackage{amsmath}
\usepackage{graphicx}
\usepackage{enumitem}
\setlist{nosep}
\setlist[enumerate,1]{label=(\arabic*)}
\usepackage{amsfonts}
\usepackage{pictex}
\usepackage{amssymb}
\usepackage{refcount}

\newtheorem{theorem}{Theorem}
\newtheorem{prop}[theorem]{Proposition}
\newtheorem{lemma}[theorem]{Lemma}
\newtheorem{coro}[theorem]{Corollary}

\newtheorem{obs}[theorem]{Observation}

\newcounter{claimprefix}
\newtheorem{claim}{Claim}[claimprefix]

\newcommand*{\claimproof}{Proof of the Claim}
\newenvironment{proofcl}[1][\claimproof]{\begin{proof}[#1]
}{\end{proof}}

\theoremstyle{definition}
\newtheorem{definition}{Definition}

\newcommand{\sm}{\setminus}
\newcommand{\claw}{K_{1,3}}

\newcommand{\AAA}{\mathcal A}
\newcommand{\CC}{\mathcal C}
\newcommand{\FF}{\mathcal F}
\newcommand{\HH}{\mathcal H}
\newcommand{\GG}{\mathcal G}
\newcommand{\II}{\mathcal I}
\newcommand{\PP}{\mathcal P}
\newcommand{\RR}{\mathcal R}

\newcommand{\XX}{\mathcal X}
\newcommand{\OO}{\Omega}

\newcommand{\splitatcommas}[1]{%
  \begingroup
  \begingroup\lccode`~=`, \lowercase{\endgroup
    \edef~{\mathchar\the\mathcode`, \penalty0 \noexpand\hspace{0pt plus 1em}}%
  }\mathcode`,="8000 #1%
  \endgroup
}

\newcommand{\vl}{l\kern-0.035cm\char39\kern-0.03cm}
\begin{document}
\title{Forbidden induced pairs for perfectness and $\omega$-colourability of graphs}
\author{Maria Chudnovsky\thanks{Department of Mathematics, Princeton University,
U.S.A.\newline
E-mail: {\tt mchudnov@math.princeton.edu}.}
\and Adam Kabela\thanks{Faculty of Applied Sciences, University of West Bohemia,
Czech Republic. Previous affiliation: Faculty of Informatics, Masaryk University, Czech Republic.\newline
E-mail: {\tt kabela@kma.zcu.cz}.}
\and
Binlong Li\thanks{School of Mathematics and Statistics, Northwestern Polytechnical University, China.\newline
E-mail: {\tt libinlong@mail.nwpu.edu.cn}.}
\and
Petr Vr\'{a}na\thanks{Faculty of Applied Sciences, University of West Bohemia,
Czech Republic.\newline
E-mail: {\tt vranap@kma.zcu.cz}.}
}
\date{}

\maketitle
\begin{abstract}
We characterise the pairs of graphs $\{ X, Y \}$
such that all $\{ X, Y \}$-free graphs
(distinct from $C_5$) are perfect.
Similarly,
we characterise pairs $\{ X, Y \}$ such that all $\{ X, Y \}$-free graphs
(distinct from $C_5$) are $\omega$-colourable
(that is, their chromatic number is equal to their clique number).
More generally,
we show characterizations of pairs $\{ X, Y \}$
for perfectness and $\omega$-colourability
of all connected $\{ X, Y \}$-free graphs which are
of independence at least~$3$,
distinct from an odd cycle,
and of order at least $n_0$,
and similar characterisations subject to each subset of these additional constraints.
(The classes are non-hereditary and the characterisations for perfectness
and $\omega$-colourability are different.)
We build on recent results of Brause et al. on $\{ K_{1,3}, Y \}$-free graphs,
and we use Ramsey's Theorem and the Strong Perfect Graph Theorem as main tools.
We relate the present characterisations to known results on forbidden pairs for
$\chi$-boundedness and deciding $k$-colourability in polynomial time.
\end{abstract}


\section{Introduction}
\label{sIntro}
The study of forbidden induced subgraphs in relation to perfectness of graphs goes back to  
Berge~\cite{B} who conjectured that perfect graphs can be characterised 
by forbidding all induced odd cycles of length at least five and their complements;
this was later proven by
the first author and Robertson, Seymour and Thomas~\cite{CRST}
(and the result is referred to as the Strong Perfect Graph Theorem).
For the sake of brevity, we ought to say that
graph perfectness is one of the classical topics in graph theory
and refer the reader to
a monograph by Ram\'{i}rez-Alfons\'{i}n and Reed~\cite{JR}
and to surveys by Hougardy~\cite{H}
and by Roussel, Rusu and Thuillier~\cite{RRT}.

In the spirit of graph perfectness,
Gy\'arf\'as~\cite{G} initiated the study of hereditary classes of graphs whose 
chromatic number can be bounded by a function of their clique number.
%
This topic (referred as $\chi$-boundedness) is also widely investigated,
in particular for classes of graphs defined by forbidden induced subgraphs.
One of the major open conjectures in this area, stated by Gy\'arf\'as~\cite{G}
and by Sumner~\cite{Sumner}, asserts that
for every forest $F$, the class of all $F$-free graphs is $\chi$-bounded.
We refer the reader to surveys by Schiermeyer and Randerath~\cite{SR}
and by Scott and Seymour~\cite{SandS}
(see also the recent result of Bonamy and Pilipczuk~\cite{BP}).

Similarly, conditions on forbidden induced subgraphs are investigated in relation
to the computational complexity of colourability.
In particular, Kr\'{a}{\vl}, Kratochv\'{i}l, Tuza and Woeginger~\cite{KKTW}
characterised graphs $H$ such that,
given an $H$-free graph and an integer $k$,
the question whether the graph is $k$-colourable can be decided in polynomial time
(namely, for $H$ chosen as an induced subgraph of $K_1 \cup P_3$ or $P_4$,
and for every other graph $H$ the problem is NP-complete),
and they initiated the investigation of forbidden pairs.
More results on the topic can be found surveyed by
Golovach, Johnson, Paulusma and Song~\cite{GJPS}
who also asked for the complete classification of
forbidden pairs in relation to the complexity of colourability.
This question is still open.
For perfect graphs the fact that $k$-colourability 
is decidable in polynomial time 
follows from the study of Gr\"otschel, Lov\'asz and Schrijver~\cite{GLS}.

In relation to the present paper,
we recall that every $P_4$-free graph is perfect by Seinsche~\cite{Seinsche}
(it also follows from the Strong Perfect Graph Theorem).
On the other hand,  
if the chromatic number is bounded by a linear function of the clique number
for the class of graphs defined by a forbidden induced subgraph $H$,
then $H$ is an induced subgraph of $P_4$
(for instance, this follows by combining the result of Erd\H{o}s~\cite{E}
who showed that there are graphs of arbitrarily large girth and chromatic number
and the result Brause, Randerath, Schiermeyer and Vumar~\cite{BRSV}
who showed that there is no linear binding function
for the class of $\{ 3K_1, 2K_2 \}$-free graphs).
On the other hand,  
particular forbidden induced subgraphs (and sets of subgraphs)
are known to give polynomial binding functions,
and there is a number of pairs of forbidden induced subgraphs giving linear binding functions.
The results and open questions
of this type can be found surveyed in~\cite{SR};
in particular, the question of characterising
the pairs of connected forbidden subgraphs
giving that the chromatic number is at most the clique number plus one.
It seems natural to also ask about the improvement of a binding function
when restricting a class by additional constraints.
For instance,
the first author and Seymour~\cite{CS} showed that
the chromatic number of a connected $\claw$-free graph of independence at least $3$
is at most twice its clique number;
whereas $\claw$-free graphs in general
admit a polynomial binding function
(for instance, see~\cite{SR}).
Considering forbidden pairs $\{ \claw, Y \}$,
Brause et al.~\cite{BHKRSV} showed that, depending on the choice of~$Y$,
the class of all connected $\{ \claw, Y \}$-free graphs distinct from an odd cycle 
(and of independence at least $3$)
either consists of perfect graphs
or contains infinitely many graphs which are not $\omega$-colourable
(see Theorem~\ref{alpha3} in Section~\ref{sOne}).
The dichotomic nature of this result is non-trivial
(since the classes are non-hereditary,
perfectness and $\omega$-colourability are different concepts),
and it motivates the general question of forbidden pairs $\{ X, Y \}$,
which is answered by the present study.


As the main results,
we characterise the pairs of forbidden induced subgraphs
in relation to perfectness and $\omega$-colourability 
for classes of graphs restricted by additional constraints.
The basic characterisations concern pairs $\{ X, Y \}$ such that 
every $\{ X, Y \}$-free graph (distinct from $C_5$) is perfect,
and pairs
such that every $\{ X, Y \}$-free graph (distinct from $C_5$) is $\omega$-colourable
(see Theorem~\ref{main} in Section~\ref{sMain}).
More generally,
we present characterisations of pairs $\{ X, Y \}$
giving perfectness ($\omega$-colourability)
of all connected $\{ X, Y \}$-free graphs of independence at least $3$
distinct from an odd cycle,
and analogous characterisations subject to
each subset of the additional constraints
(stated formally in Theorem~\ref{mainNoExceptions}).
Furthermore,
we characterise pairs $\{ X, Y \}$ in relation to
perfectness ($\omega$-colourability) of all but finitely many
$\{ X, Y \}$-free graphs satisfying the constraints
(see Theorem~\ref{mainFiniteExceptions}).
In other words, there is $n_0$ (depending on $\{ X, Y \}$) such that
all considered $\{ X, Y \}$-free graphs on at least $n_0$ vertices
are perfect ($\omega$-colourable).
The present characterisations related to perfectness
are, in fact, different from those related to $\omega$-colourability
due to the consideration of the additional constraints.
(The constraints are discussed in Section~\ref{sOne}.)


We note that the present conditions on forbidden pairs
(without the additional constraints on a class)
were studied in relation to $\chi$-boundedness and to computational complexity of colourability.
In fact, under each of the conditions (or its generalisation), 
the class is known to be $\chi$-bounded 
and the $k$-colourability is decidable in polynomial time
(see the known results surveyed in~\cite{SR} and~\cite[Theorem 21 (ii)]{GJPS},
respectively).
We remark that for some pairs,
the $\chi$-boundedness and the fact that $k$-colourability is decidable in polynomial time
can also be deduced 
using the present results
(for details, see Section~\ref{sOne}).
In this sense, Theorems~\ref{mainNoExceptions} and~\ref{mainFiniteExceptions}
show that these forbidden pairs yield a stronger property.
%


The paper is structured as follows.
In~Section~\ref{sPre}, we recall basic notation.
In~Section~\ref{sMain}, we state the main results on the forbidden pairs
(see Theorems~\ref{main}, \ref{mainNoExceptions} and~\ref{mainFiniteExceptions}).
In addition, we use the fact that a collection of pairs of graphs
can be represented by a graph and we provide figures depicting the collections.
In Section~\ref{sOne},
we discuss the additional constraints and the pairings given by the characterisations,
and we comment on the connections to known results.
The key steps for proving the main results are broken down
in Sections~\ref{sForb}, \ref{sPerf}, \ref{sOmega} and~\ref{sFam}.
In Section~\ref{sForb}, we show several structural lemmas
for classes defined by particular forbidden induced subgraphs.
We use the classical result of Ramsey~\cite{R} (recalled in Theorem~\ref{Ramsey})
and the characterisation of connected $Z_1$-free graphs by Olariu~\cite{O}
(recalled in Lemma~\ref{olariu}).
In Section~\ref{sPerf},
we show a detailed statement on how conditions on forbidden pairs imply perfectness.
Using the structural lemmas shown in the previous section, the perfectness
follows either directly (for most pairs) or with the aid of the Strong Perfect Graph Theorem
(recalled in Theorem~\ref{tSPGT}).
Similarly in Section~\ref{sOmega},
we show how conditions on forbidden pairs imply $\omega$-colourability.
We use the Strong Perfect Graph Theorem and one of the structural lemmas
(shown in Section~\ref{sForb})
and also Ramsey's Theorem.
In Section~\ref{sFam}, we exclude all possible remaining pairs.
In particular, we recall that the pairs containing $\claw$ were resolved in~\cite{BHKRSV},
and we eliminate the remaining cases
with the help of several constructions of families of graphs which are 
not perfect (not $\omega$-colourable)
and structural properties observed in Section~\ref{sForb}.
Finally in Section~\ref{sProving}, we combine the ingredients
(assembled in Sections~\ref{sOne} and \ref{sPerf}, \ref{sOmega} and~\ref{sFam})
and prove the main results.

\section{Notation}
\label{sPre}
In this short section, we recall basic definitions and concepts.
We refer the reader to~\cite{BM} for additional notation
and for a detailed introduction to the concepts of subgraphs and graph colourings.

We recall that a graph $G$ is \emph{$k$-colourable}
if the vertices of $G$ can be coloured with $k$ colours
so that adjacent vertices obtain distinct colours;
and the smallest integer $k$ with this property
is called the \emph{chromatic number} of $G$, and denoted $\chi(G)$.
We let $\omega(G)$ denote the \emph{clique number} of a graph $G$,
that is, the number of vertices of a maximum complete subgraph of $G$.
Clearly, $\chi(G) \geq \omega(G)$ for every graph $G$.
For the sake of simplicity,
we say that a graph $G$ is \emph{$\omega$-colourable} if $\chi(G) = \omega(G)$.
We recall that a graph $G$ is \emph{perfect} if $\chi(G')=\omega(G')$
for every induced subgraph $G'$ of $G$.

Given a family $\HH$ of graphs, we say that a graph $G$ is \emph{$\HH$-free}
if $G$ contains no member of $\HH$ as an induced subgraph
(and we write $\{ H_1, H_2, \dots, H_k \}$-free for a family $\{ H_1, H_2, \dots, H_k \}$,
and $H$-free as a short for $\{ H \}$-free).
We let $K_n, P_n, C_n$ denote the complete graph, the path, the cycle on $n$ vertices, respectively.
We let $K_{n_1, \dots, n_k}$ denote the complete $k$-partite graph
whose partities are of sizes $n_1, \dots, n_k$.
We let $\claw^+$ denote the graph obtained from $\claw$ by subdividing an edge,
and $D$ the graph obtained from $K_4$ by removing an edge,
and $Z_1$ the graph obtained from $K_3$ by adding a pendant edge,
and $Z_2$ the graph obtained from $Z_1$ by subdividing the pendant edge
(see graphs $\claw^+, D, Z_1$ and $Z_2$ depicted in the bottom part of Figure~\ref{figPairsX};
the graphs are commonly referred to as chair, diamond, paw and hammer, respectively).

We let $G \cup H$ denote the disjoint union of graphs $G$ and $H$,
and we write $k G$ for the disjoint union of $k$ copies of $G$.
We say that $G$ is of independence at least $k$ if $G$ contains induced $k K_1$;
and the largest such integer $k$ is called the \emph{independence number} of $G$. 
We let $\overline{G}$ denote the complement of a graph $G$
that is, the graph on the same vertex set as $G$ whose vertices are adjacent if and only if
the corresponding vertices of $G$ are not.
(For instance see graph $\overline{K_3 \cup P_4}$ depicted in Figure~\ref{figPairsX}.)

\section{Main results}
\label{sMain}

We present characterisations of the pairs of forbidden induced subgraphs
in relation to perfectness and $\omega$-colourability of graphs,
see Theorems~\ref{main}, \ref{mainNoExceptions} and~\ref{mainFiniteExceptions}.
The pairs of interest are collected in Definitions~\ref{d1}, \ref{d2} and~\ref{d3},
and the class notation is given in Definition~\ref{dc}.
We invite the reader to consult Figures~\ref{f1},~\ref{figPairsX},~\dots,~\ref{fc}.

\begin{definition}\label{d1}
Consider graphs $3K_1, K_1 \cup P_3, 2K_1 \cup K_2, K_3, Z_1$ and $D$
(defined in Section~\ref{sPre}).
Let $\PP_1$ be the collection of all pairs $\XX$ of graphs
such that at least one of the following conditions is satisfied.
\begin{itemize}
\item
At least one of the graphs of $\XX$ is an induced subgraph of $P_4$.
\item
$\XX$ is either $\{ 3K_1, K_3 \}$ or $\{ 3K_1, Z_1 \}$ or $\{ 3K_1, D \}$.
\item
$\XX$ is either $\{ K_1 \cup P_3, K_3 \}$ or $\{ K_1 \cup P_3, Z_1 \}$ or $\{ K_1 \cup P_3, D \}$.
\item
$\XX$ is either $\{ 2K_1 \cup K_2, K_3 \}$ or $\{ 2K_1 \cup K_2, Z_1 \}$.
\end{itemize}
Furthermore, let $\OO_1$ be the collection consisting of $\{ 2K_1 \cup K_2, D \}$
and all pairs of $\PP_1$.
\end{definition}
The following are the basic characterisations
(see also Figure~\ref{f1}).
\begin{theorem}
\label{main}
Let $\XX$ be a pair of graphs and
$\PP_1$ and $\OO_1$ be the collections described in Definition~\ref{d1}.
Then every $\XX$-free graph (distinct from $C_5$) is perfect
if and only if
$\XX$ belongs to $\PP_1$.
Similarly, every $\XX$-free graph (distinct from $C_5$) is $\omega$-colourable
if and only if
$\XX$ belongs to $\OO_1$.
\end{theorem}
\begin{figure}[h!]
    \centering
    \includegraphics[scale=0.82]{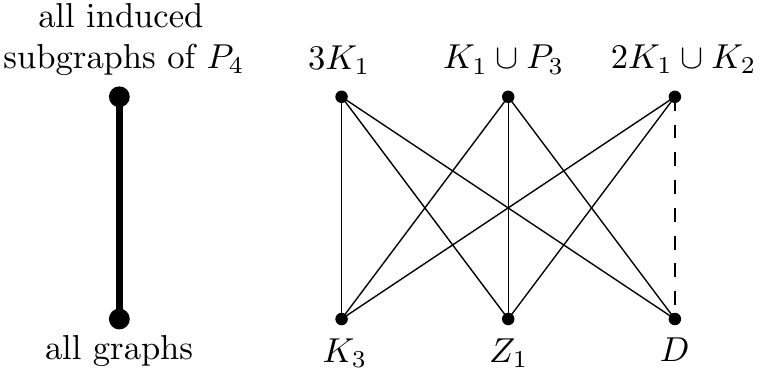}
    \caption{Depicting collections $\PP_1$ and $\OO_1$
    described in Definition~\ref{d1}.
    The graph in the picture represents collection $\OO_1$ of pairs of graphs as follows.
    Every vertex represents a graph and two vertices are adjacent if and only if
    the corresponding graphs form a pair.
    (Each of the bold vertices represents a family of graphs
    and the bold edge represents the collection of all pairs given by the families.)
    The dashed edge represents the pair belonging to $\OO_1$ but not to $\PP_1$.
    (We note that $3K_1, K_1 \cup P_3, 2K_1 \cup K_2$ is the complement of 
	$K_3, Z_1, D$, respectively.)
    }
    \label{f1}
\end{figure}
As the main results, we extend the characterisations of Theorem~\ref{main}
in relation to classes of graphs given by particular sets of additional constraints.
The extended characterisations are presented in Theorems~\ref{mainNoExceptions}
and~\ref{mainFiniteExceptions}.
In particular,
we note that Theorem~\ref{main}
is implied by the combination of items (1) and (6) of Theorem~\ref{mainNoExceptions}.

\begin{definition}\label{dc}
For the sake of brewity, we use the following notations.
\begin{itemize}
\item
$\GG$ is the class of all graphs,
\item
$\GG_{5}$ is the class of all graphs distinct from $C_5$,
\item
$\GG_{o}$ is the class of all graphs distinct from an odd cycle,
\item
$\GG_{c}$ is the class of all connected graphs,
\item
$\GG_{c, 5}$ is the class of all connected graphs distinct from $C_5$,
\item
$\GG_{\alpha}$ is the class of all graphs of independence at least $3$,
\item
$\GG_{o, \alpha}$ is the class of all graphs of independence at least $3$
distinct from an odd cycle,
\item
$\GG_{c, \alpha}$ is the class of all connected graphs of independence at least $3$,
\item
$\GG_{c, o}$ is the class of all connected graphs distinct from an odd cycle,
\item
$\GG_{c, o, \alpha}$ is the class of all connected graphs of independence at least $3$
distinct from an odd cycle.
\end{itemize}
\end{definition}

\begin{figure}[hb!]
    \centering
    \includegraphics[scale=0.53]{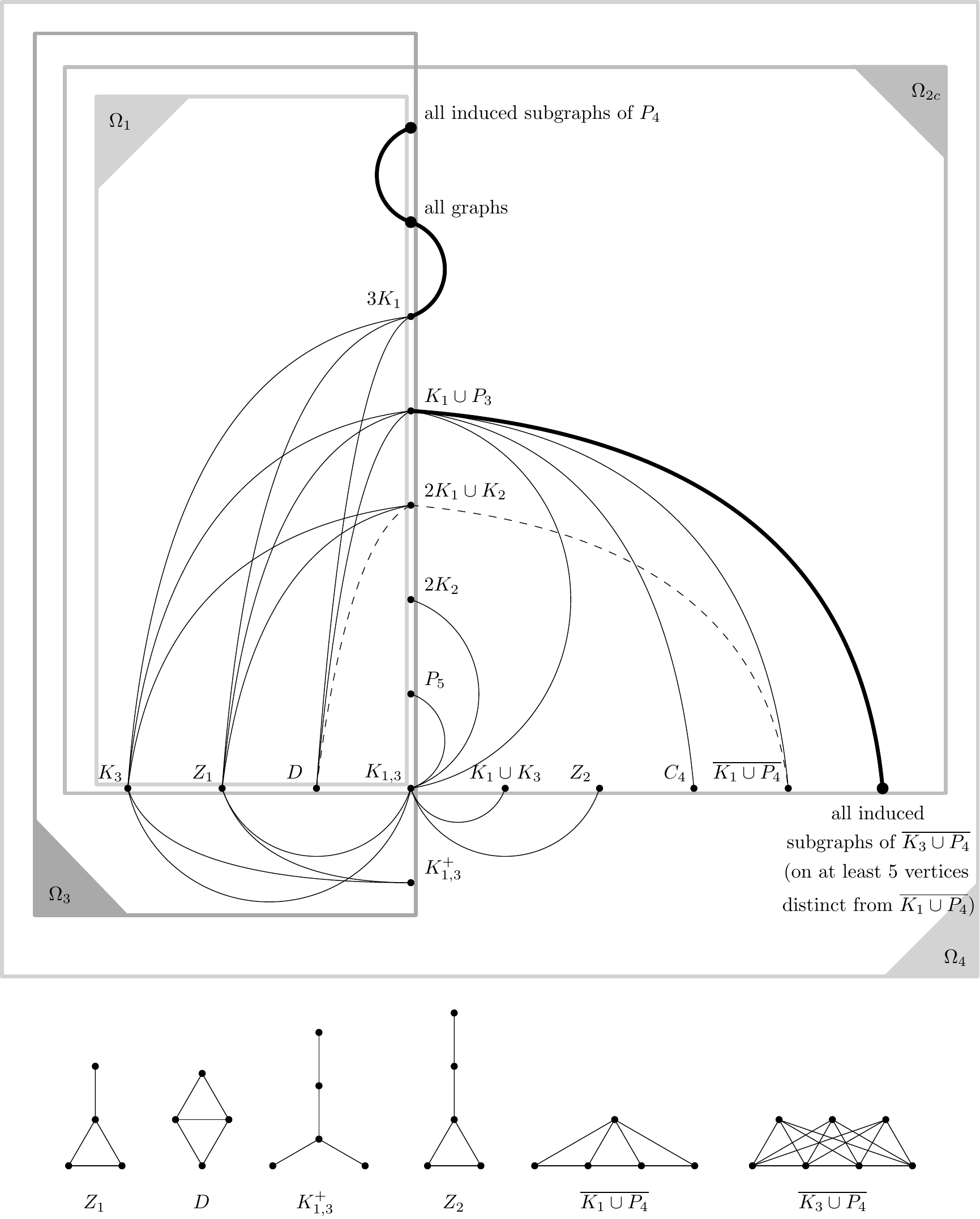}
    \caption{Collections $\OO_1, \OO_{2c}, \OO_3, \OO_4$ and $\PP_1, \PP_{2c}, \PP_3, \PP_4$
    of pairs of graphs described in Definition~\ref{d2} (top part of the figure),
    and graphs
    $Z_1, D, \claw^+, Z_2, \overline{K_1 \cup P_4}$ and $\overline{K_3 \cup P_4}$
    (bottom part).
    In the top part,
    the vertices represent graphs
    and the edges represent pairs of graphs
    (similarly as in Figure~\ref{f1}).
    The grey rectangles represent the belonging of the pairs
    to collections $\OO_1, \OO_{2c}, \OO_3, \OO_4$,
    and also to collections $\PP_1, \PP_{2c}, \PP_3, \PP_4$
    when not considering the pairs depicted by dashed edges.
    }
    \label{figPairsX}
\end{figure}

\begin{definition}\label{d2}
Consider graphs
$\splitatcommas{3K_1, 2K_1 \cup K_2, K_1 \cup P_3, 2K_2, \claw, \claw^+, P_5, K_3, K_1 \cup K_3, D, Z_1, Z_2, \overline{K_1 \cup P_4}}$ and $\overline{K_3 \cup P_4}$
defined in Section~\ref{sPre}
(some of the graphs are also depicted in Figure~\ref{figPairsX}),
and collections $\PP_1$ and $\OO_1$ described in Definition~\ref{d1}.
Let $\II$ be the collection of all pairs of graphs
such that at least one member of the pair is $3K_1$.
Let $\PP_2, \PP_{2c}, \PP_3$ and $\PP_4$ be the collections of pairs of graphs
defined as follows.
\begin{itemize}
\item
$\PP_2$ consists of all pairs $\XX$ such that $\XX$ belongs to $\PP_1$ or to $\II$
or such that one member of $\XX$ is $K_1 \cup P_3$ and the other is
an induced subgraph of $\overline{K_3 \cup P_4}$.
\item
$\PP_{2c}$ consists of all pairs of $\PP_2$ and $\{ \claw, 2K_2 \}$ and $\{ \claw, P_5 \}$.
\item
$\PP_{3}$ consists of all pairs of $\PP_1$ and
$\{ \claw, K_3 \}$ and
$\{ \claw, Z_1 \}$ and
$\{ \claw^+, K_3 \}$ and
$\{ \claw^+, Z_1 \}$.
\item
$\PP_{4}$ consists of all pairs of $\PP_{2c}$ and $\PP_3$ and
$\{ \claw, K_1 \cup K_3 \}$ and
$\{ \claw, Z_2 \}$.
\end{itemize}
Similarly,
let $\OO_2, \OO_{2c}, \OO_3, \OO_4$ be the collections of pairs of graphs
defined as follows.
\begin{itemize}
\item
$\OO_2$ consists of all pairs of $\PP_2$ and
$\{ 2K_1 \cup K_2, D \}$ and 
$\{ 2K_1 \cup K_2, \overline{K_1 \cup P_4} \}$.
\item
$\OO_{2c}$ consists of all pairs of $\PP_{2c}$ and
$\{ 2K_1 \cup K_2, D \}$ and 
$\{ 2K_1 \cup K_2, \overline{K_1 \cup P_4} \}$.
\item
$\OO_3$ consists of all pairs of $\PP_3$ and 
$\{ 2K_1 \cup K_2, D \}$.
\item
$\OO_4$ consists of all pairs of $\PP_4$ and
$\{ 2K_1 \cup K_2, D \}$ and 
$\{ 2K_1 \cup K_2, \overline{K_1 \cup P_4} \}$.
\end{itemize}
\end{definition}
\begin{figure}[h!]
    \centering
    \includegraphics[scale=0.7]{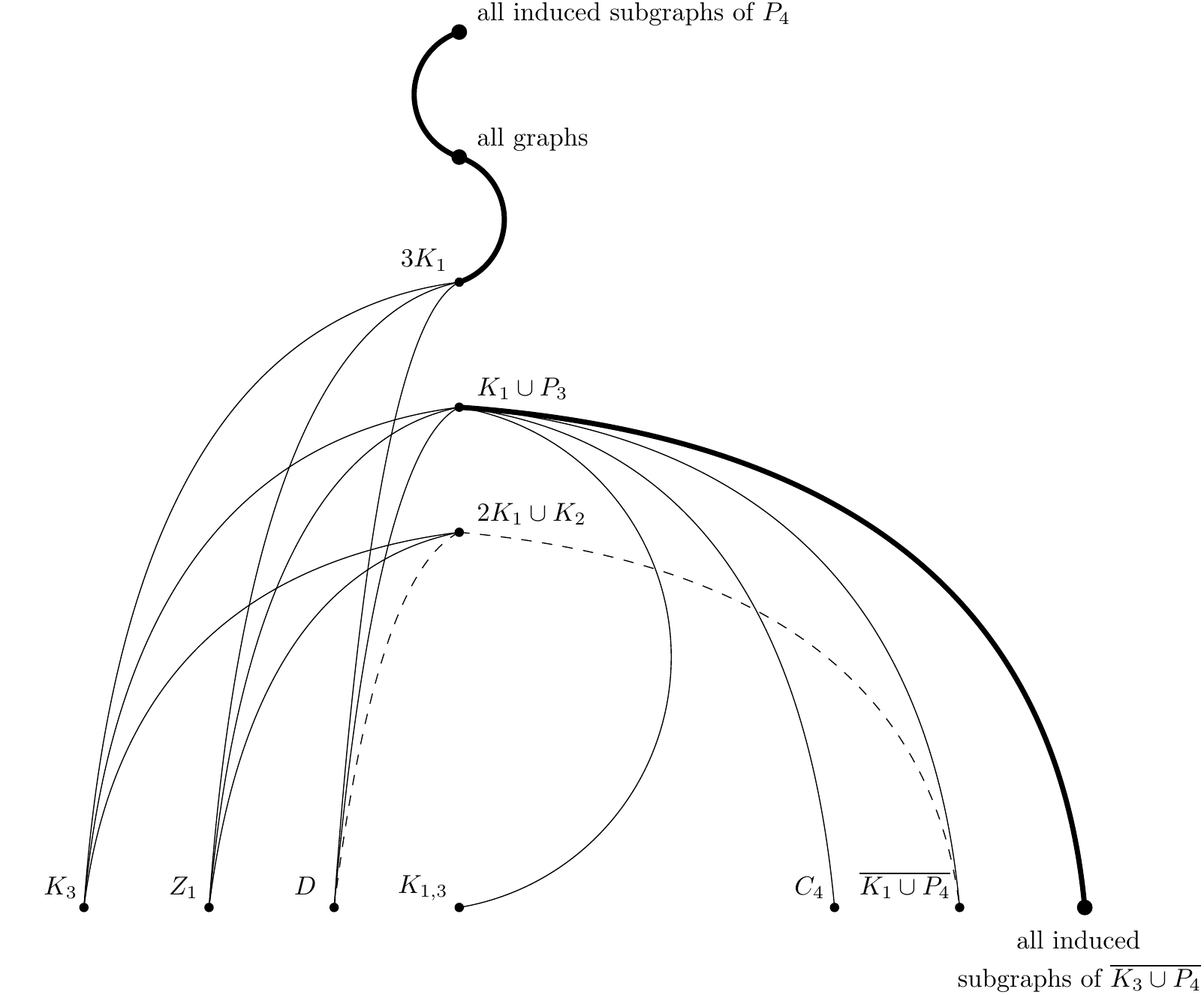}
    \caption{Collections $\OO_2$ and $\PP_2$ described in Definition~\ref{d2}.
    The vertices represent graphs and the edges represent pairs of graphs
    and the dashed edges represent the pairs belonging to $\OO_2$ but not to $\PP_2$
    (similarly as in Figure~\ref{f1}).
    }
    \label{figPairsX2}
\end{figure}

The first of the main results is the following
(see also Figures~\ref{figPairsX}, \ref{figPairsX2} and~\ref{fc}
and the commentary in Section~\ref{sOne}).

\begin{theorem}
\label{mainNoExceptions}
Let $\XX$ be a pair of graphs,
and consider the class notation of Definition~\ref{dc}
and the collections described in Definitions~\ref{d1} and~\ref{d2}.
For each of the following choices of a class,
every $\XX$-free graph of the class is perfect
if and only if $\XX$ belongs to the corresponding collection.
\begin{enumerate}
\item
For each of $\GG_5$, $\GG_{o}$, $\GG_{c,5}$, the collection is $\PP_1$.
\item
For each of $\GG_{\alpha}$, $\GG_{o,\alpha}$, it is $\PP_2$.
\item
For $\GG_{c, \alpha}$, it is $\PP_{2c}$.
\item
For $\GG_{c, o}$, it is $\PP_3$.
\item
For $\GG_{c, o, \alpha}$, it is $\PP_4$.
\end{enumerate}
Similarly,
every $\XX$-free graph of the class is $\omega$-colourable
if and only if $\XX$ belongs to the collection as follows.
\begin{enumerate}
\setcounter{enumi}{5}
\item
For each of $\GG_5$, $\GG_{o}$, $\GG_{c,5}$, it is $\OO_1$.
\item
For each of $\GG_{\alpha}$, $\GG_{o,\alpha}$, it is $\OO_2$.
\item
For $\GG_{c, \alpha}$, it is $\OO_{2c}$.
\item
For $\GG_{c, o}$, it is $\OO_3$.
\item
For $\GG_{c, o, \alpha}$, it is $\OO_4$.
\end{enumerate}
\end{theorem}

\begin{figure}[hb!]
    \centering
    \includegraphics[scale=0.58]{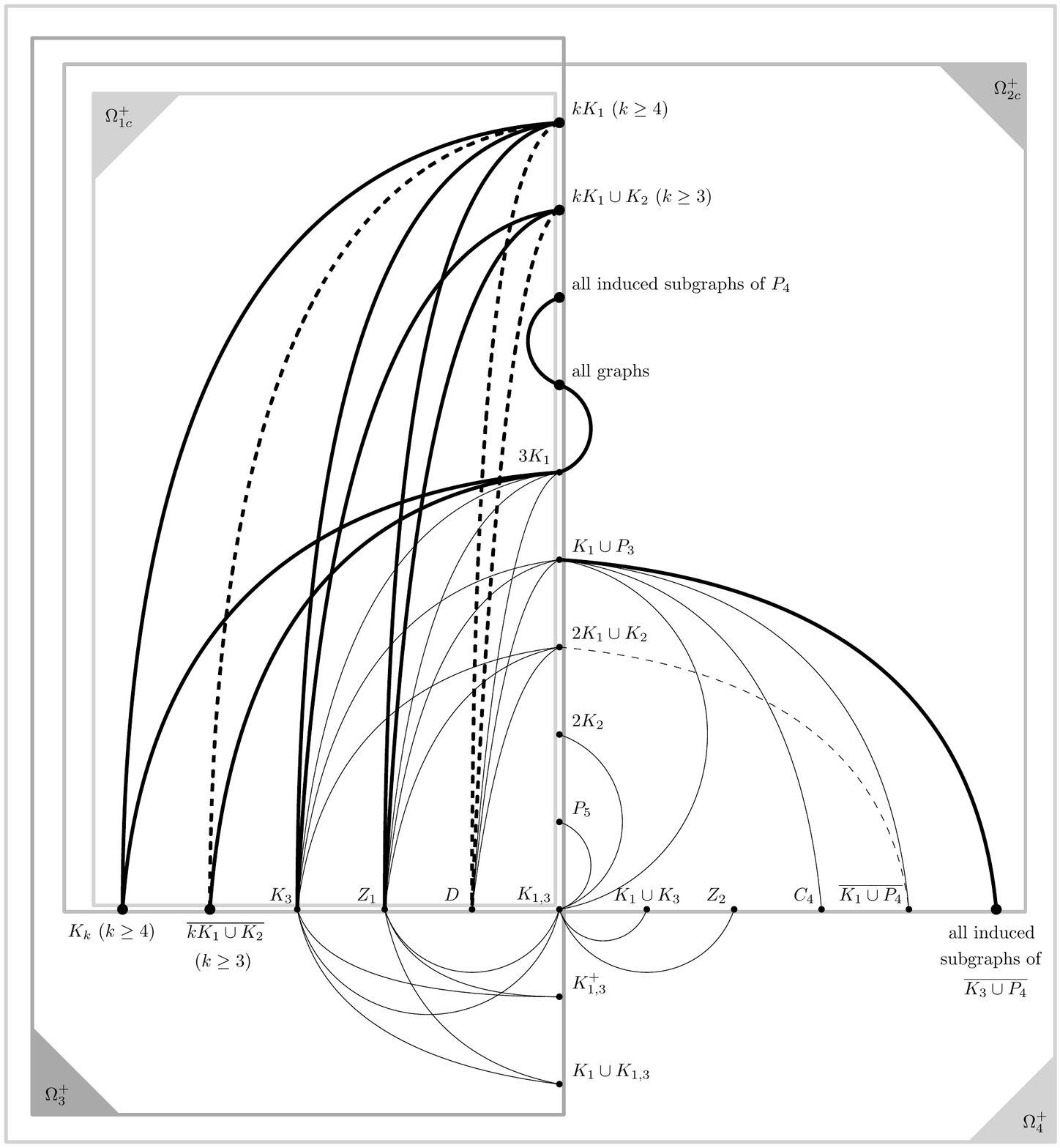}
    \caption{Collections
    $\OO_{1c}^+, \OO_{2c}^+, \OO_3^+, \OO_4^+$ and $\PP_{1c}^+, \PP_{2c}^+, \PP_3^+, \PP_4^+$
    described in Definition~\ref{d3}.
    The edges represent pairs of graphs and
    the grey rectangles represent the belonging to collections
    (similarly as in Figure~\ref{figPairsX}).
    }
    \label{figPairsXplus}
\end{figure}

In the final definition,
we extend the collections as follows.
\begin{definition}\label{d3}
Consider collections $\PP_1, \dots, \PP_4$ and $\PP_{2c}$
described in Definitions~\ref{d1} and~\ref{d2}.
Let $\RR$ be the collection of all pairs of the form
$\{k K_1, K_\ell \}$
where $k \geq 4$ and $\ell \geq 3$.
Let $\AAA_{\PP}, \AAA_1, \AAA_c, \AAA_3$ and $\PP_1^+, \dots, \PP_4^+$ and $\PP_{1c}^+, \PP_{2c}^+$
be the collections of pairs of graphs defined as follows.
\begin{itemize}
\item
$\AAA_{\PP}$ consists of all pairs of $\RR$ and
$\{ 2K_1 \cup K_2, D \}$ and
all pairs of the form $\{ kK_1 \cup K_2, K_3 \}$
where $k \geq 3$.
\item
$\AAA_1$ consists of all pairs
$\{ 3K_1, K_{k+1} \}$ and $\{ 3K_1, \overline{ kK_1 \cup K_2 } \}$ where $k \geq 3$
(note that these pairs belong to $\II$ and thus to $\PP_2, \PP_{2c}, \PP_4$
and $\OO_2, \OO_{2c}, \OO_4$).
\item
$\AAA_c$ consists of
all pairs
$\{ (k+1)K_1, Z_1 \}$ and
$\{ kK_1 \cup K_2, Z_1 \}$
where $k \geq 3$.
\item
$\AAA_3$ consists of 
$\{ K_1 \cup \claw, K_3 \}$ and
$\{ K_1 \cup \claw, Z_1 \}$.
\item
$\PP_1^+$ consists of all pairs of $\PP_1$ and $\AAA_{\PP}$ and $\AAA_1$.
\item
$\PP_{1c}^+$ consists of all pairs of $\PP_1$ and $\AAA_{\PP}$ and $\AAA_1$ and $\AAA_c$.
\item
$\PP_2^+$ consists of all pairs of $\PP_2$ and $\AAA_{\PP}$.
\item
$\PP_{2c}^+$ consists of all pairs of $\PP_{2c}$ and $\AAA_{\PP}$ and $\AAA_c$.
\item
$\PP_3^+$ consists of all pairs of $\PP_3$ and $\AAA_{\PP}$ and $\AAA_1$ and $\AAA_c$ and $\AAA_3$.
\item
$\PP_4^+$ consists of all pairs of $\PP_4$ and $\AAA_{\PP}$ and $\AAA_c$ and $\AAA_3$.
\end{itemize}
Similarly, let $\AAA_{\Omega}$ and $\OO_1^+, \dots, \OO_4^+$ and $\OO_{1c}^+, \OO_{2c}^+$
be the collections of pairs of graphs defined as follows.
\begin{itemize}
\item
$\AAA_{\Omega}$ consists of all pairs of $\AAA_{\PP}$ and $\AAA_c$
(except for the pair $\{ 2K_1 \cup K_2, D \}$ which is not necessary since it belongs
to $\OO_1$ already)
and all pairs of the form
$\{ (k+1) K_1, D \}$ and
$\{ k K_1 \cup K_2, D \}$
where $k \geq 3$
and
$\{ k K_1, \overline{ \ell K_1 \cup K_2 } \}$
where $k \geq 4$ and $\ell \geq 3$.
\item
$\OO_1^+$ consists of all pairs of $\OO_1$ and $\AAA_{\Omega}$ and $\AAA_1$.
(In other words, it consists of all pairs of $\PP_{1c}^+$ and all pairs of the form
$\{ (k+1) K_1, D \}$ and
$\{ k K_1 \cup K_2, D \}$
and
$\{ (k+1) K_1, \overline{ \ell K_1 \cup K_2 } \}$
where $k \geq 3$ and $\ell \geq 3$.)
\item
$\OO_{1c}^+$ is the same as $\OO_1^+$.
\item
$\OO_2^+$ consists of all pairs of $\OO_2$ and $\AAA_{\Omega}$.
\item
$\OO_{2c}^+$ consists of all pairs of $\OO_{2c}$ and $\AAA_{\Omega}$.
\item
$\OO_3^+$ consists of all pairs of $\OO_3$ and $\AAA_{\Omega}$ and $\AAA_1$ and $\AAA_3$.
\item
$\OO_4^+$ consists of all pairs of $\OO_4$ and $\AAA_{\Omega}$ and $\AAA_3$.
\end{itemize}
\end{definition}

\begin{figure}[h!]
    \centering
    \includegraphics[scale=0.6]{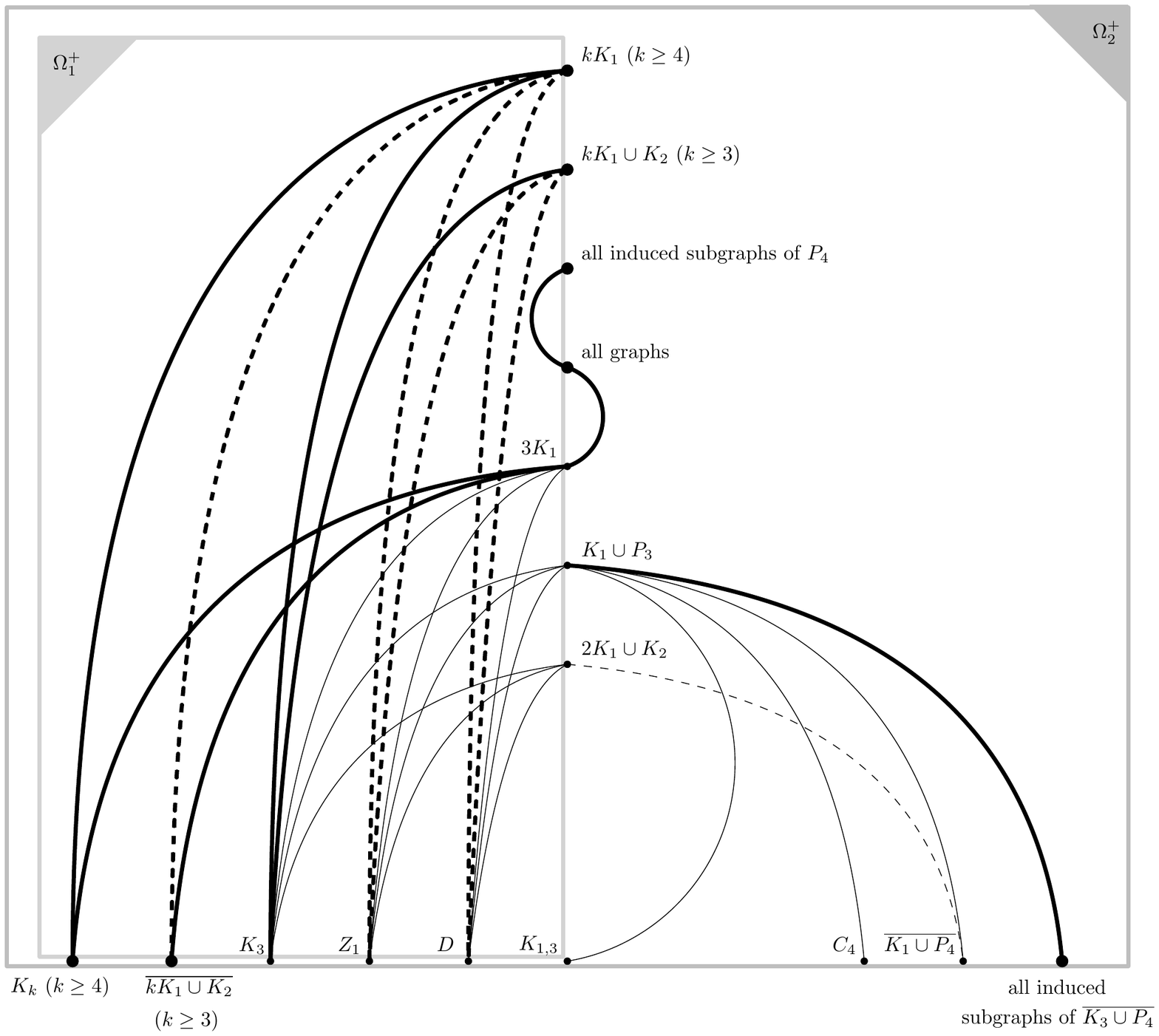}
    \caption{Collections $\OO_1^+, \OO_2^+$ and $\PP_1^+, \PP_2^+$ described in Definition~\ref{d3}.
    }
    \label{figPairsX12plus}
\end{figure}

The second of the main results is the following
(see also Figures~\ref{figPairsXplus}, \ref{figPairsX12plus} and~\ref{fc}).
\begin{theorem}
\label{mainFiniteExceptions}
Let $\XX$ be a pair of graphs,
and consider the class notation of Definition~\ref{dc}
and the collections described in Definition~\ref{d3}.
For each of the following choices of a class,
the set of all non-perfect $\XX$-free graphs in the class is finite
if and only if $\XX$ belongs to the corresponding collection.
\begin{enumerate}
\item
For each of $\GG$, $\GG_{o}$, the collection is $\PP_1^+$.
\item
For $\GG_{c}$, it is $\PP_{1c}^+$.
\item
For each of $\GG_{\alpha}$, $\GG_{o,\alpha}$, it is $\PP_2^+$.
\item
For $\GG_{c, \alpha}$, it is $\PP_{2c}^+$.
\item
For $\GG_{c, o}$, it is $\PP_3^+$.
\item
For $\GG_{c, o, \alpha}$, it is $\PP_4^+$.
\end{enumerate}
Similarly,
the class contains only finitely many $\XX$-free graphs
which are not $\omega$-colourable
if and only if $\XX$ belongs to the collection.
\begin{enumerate}
\setcounter{enumi}{6}
\item
For each of $\GG$, $\GG_{o}$, $\GG_{c}$, it is $\OO_1^+$.
\item
For each of $\GG_{\alpha}$, $\GG_{o,\alpha}$, it is $\OO_2^+$.
\item
For $\GG_{c, \alpha}$, it is $\OO_{2c}^+$.
\item
For $\GG_{c, o}$, it is $\OO_3^+$.
\item
For $\GG_{c, o, \alpha}$, it is $\OO_4^+$.
\end{enumerate}
\end{theorem}
\begin{figure}[h!]
    \centering
    \includegraphics[scale=0.5]{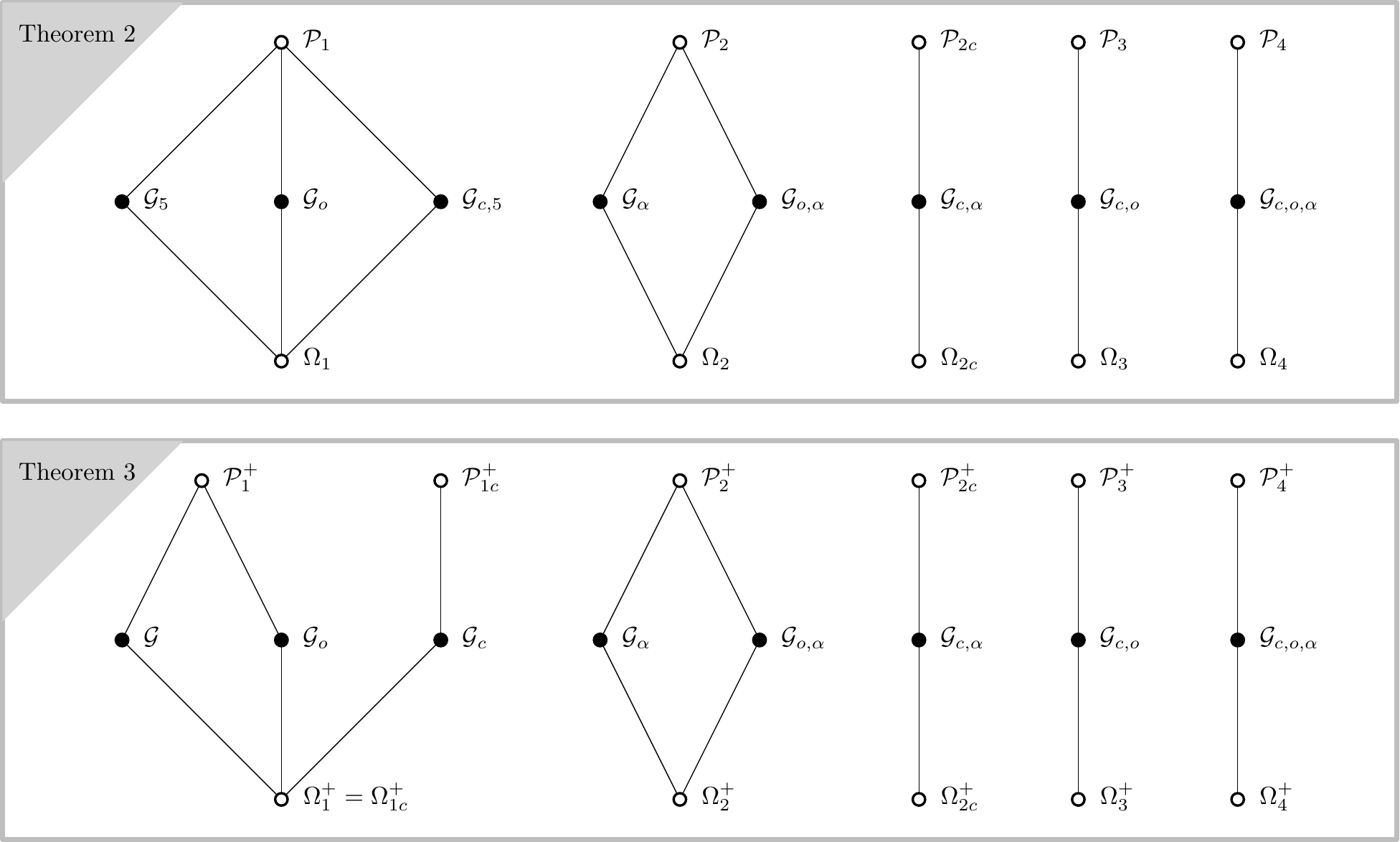}
    \caption{The bipartite graphs depicting the characterisations
    of Theorems~\ref{mainNoExceptions} and~\ref{mainFiniteExceptions}.
	Each black vertex represents a class of graphs,
	each white vertex represents a collection of pairs,
	and each edge represents a characterisation.
    }
    \label{fc}
\end{figure}

\section{Additional commentary on the main results}
\label{sOne}
In this short section,
we comment on the pairings and the additional constraints,
and on the present characterisations in the context of known results.

The basic intuition on the forbidden pairs in Theorem~\ref{main} is as follows.
As recalled in Section~\ref{sIntro}, if $X$ or $Y$ is an induced subgraph of $P_4$,
then all $\{ X, Y \}$-free graphs are perfect.
Assuming that $\{ X, Y \}$ contains no such graph,
we note that if there are only finitely many $\{ X, Y \}$-free graphs
which are not perfect (not $\omega$-colourable), then
there are only finitely many $X$-free odd cycles and 
only finitely many $Y$-free graphs which are complements of an odd cycle
(possibly for $X$ and $Y$ swapped),
and furthermore $C_5$ is $\{ X, Y \}$-free.
In particular, we view the exclusion of $C_5$
as a natural additional constraint for the present study
and we consider classes $\GG_5$ and $\GG_{c,5}$ in Theorem~\ref{mainNoExceptions}
(rather than $\GG$ and $\GG_c$).

Strengthening this additional constraint,
we also consider graphs
of independence at least $3$ and graphs distinct from an odd cycle
(the motivation comes from~\cite{CS,BHKRSV}).
In this regard, the present study is a direct continuation of
the paper of Brause et al.~\cite{BHKRSV}
who showed the following result on forbidden pairs
(and a similar result considering induced subgraphs of $Z_1$ and $P_4$
and no constraint on the independence).

\begin{theorem}\label{alpha3}
Let $Y$ be a graph and $\mathcal{G}$ be the class of all connected $\{ \claw, Y \}$-free
graphs of independence at least $3$ distinct from an odd cycle.
If $Y$ is an induced subgraph of $P_5$ or of $Z_2$,
then all graphs of $\mathcal{G}$ are perfect.
Otherwise, there are infinitely many graphs of $\mathcal{G}$ which are not
$\omega$-colourable.
\end{theorem}

In a sense, the dichotomic nature of Theorem~\ref{alpha3} says that
`the perfect-and-only partners' for $\claw$
are $P_5$ and $Z_2$ and their induced subgraphs.
With Theorems~\ref{mainNoExceptions} and~\ref{mainFiniteExceptions} on hand,
we can see that this dichotomy also works the other way around, that is,
$P_5$ and $Z_2$ (and also $2K_2$ and $K_1 \cup K_3$) are partnered with~$\claw$
(and $P_4$).
Similarly, 
$K_1 \cup P_3$ is partnered with $\overline{K_3 \cup P_4}$ (and its induced subgraphs),
and $\overline{K_3 \cup P_4}$
(and each of its induced subgraphs on at least $5$ vertices, and $C_4$)
is partnered with $K_1 \cup P_3$
(some details on induced subgraphs of $\overline{K_3 \cup P_4}$
can be found in Observations~\ref{cK3P4on5} and~\ref{cK3P4}).
%

In addition, we recall that a graph is perfect if and only if its complement is perfect
by the result of Lov\'{a}sz~\cite{L}
(it also follows from the Strong Perfect Graph Theorem).
In particular,
we note that there is certain symmetry about collection $\PP_1$.
Specifically, a pair $\{ X, Y \}$ belongs to $\PP_1$
if and only if $\{ \overline{X}, \overline{Y} \}$ does
(see Figure~\ref{f1}).
Also, there is a similar symmetry about $\PP_1^+$.
Furthermore, we can readily obtain the characterisation of the forbidden pairs
in relation to perfectness of graphs of clique number at least $3$
by taking the collection of all pairs which are complementary
to the pairs of $\PP_2$ (and similarly for $\PP_2^+$).
For completeness, we note that the collection $\II$ 
arises as a formal consequence of the additional constraint,
and similarly for collection $\RR$.

As remarked in Section~\ref{sIntro}, the present conditions on forbidden pairs
were studied in relation to $\chi$-boundedness and deciding $k$-colourability in polynomial time,
and for some pairs these properties can also be deduced using the present results.
In particular, we note that applying the `if part' of item (10) of Theorem~\ref{mainFiniteExceptions}
to connected components of a graph yields the following corollary.
For every pair $\XX$ of $\OO_3^+$,
there exists a constant $c$ such that every $\XX$-free graph is
$(\omega + c)$-colourable.
Similarly, a combination of item (6) of Theorem~\ref{mainFiniteExceptions}
and the results of Kr\'{a}{\vl} et al.~\cite{KKTW}
and Gr\"otschel et al.~\cite{GLS} gives the following.
For every pair $\XX$ of $\PP_4^+$,
the $k$-colourablity of an $\XX$-free graph can be decided in polynomial time.
We note that, for each of these corollaries, a more general result is known 
(see~\cite{SR} and~\cite[Theorem 21 (ii)]{GJPS}).
We only mention the corollaries to put the present study in context
and to note that some forbidden pairs
(known from the investigation of $\chi$-boundedness and computational complexity of $k$-colourability),
in fact,
imply a stronger property given by the hypotheses and conclusions of
Theorems~\ref{mainNoExceptions} and~\ref{mainFiniteExceptions}.
We conclude the section with a comment on
conditions giving that the chromatic number is at most the clique number plus one
(results of this type can be found in~\cite{SR}).
Considering connected graphs (of independence at least $3$),
we note that such conditions follow from
the `if part' of items (9), (10) of Theorem~\ref{mainNoExceptions}
and items (10), (11) of Theorem~\ref{mainFiniteExceptions}.

\section{Structure given by forbidden subgraphs}
\label{sForb}
In this section, we investigate structural properties of graphs
given by particular forbidden subgraphs,
and we assemble several auxiliary statements.
In particular, we outline the structure of $\{ X, K_3 \}$-free and
$\{ X, Z_1 \}$-free graphs for several choices of $X$, 
and we characterise $\{ 2K_1 \cup K_2, \overline{K_1 \cup P_4} \}$-free
graphs containing induced $C_5$.
We start the exposition by recalling the classical result of Ramsey~\cite{R}.

\begin{theorem}[Ramsey's Theorem]\label{Ramsey}
Let $k$ and $\ell$ be positive integers.
There is an integer $r$ such that every graph
(on at least $r$ vertices)
contains induced $kK_1$ or~$K_\ell$.
\end{theorem} 

We let $R(k, \ell)$ denote the \emph{Ramsey number} corresponding to the pair $\{ k, \ell \}$,
that is,
the smallest integer $r$ satisfying Theorem~\ref{Ramsey}.
We also recall the following result of Olariu~\cite{O}.

\begin{lemma}[Olariu's Characterisation]\label{olariu}
Every connected $Z_1$-free graph is $K_3$-free or complete multipartite.
\end{lemma}

With the structural similarity of $K_3$-free and $Z_1$-free graphs on hand
(given by Lemma~\ref{olariu}),
we show several statements.

\begin{obs}
\label{bip}
Every connected $\{ \claw^+, K_3 \}$-free graph
(distinct from an odd cycle) is bipartite. 
\end{obs}
\begin{proof}
For the sake of a contradiction, we suppose that
there is a set of vertices $C$ inducing an odd cycle in $G$. 
Clearly, $G$ contains a vertex, say $x$,
not belonging to $C$ but adjacent to a vertex of $C$
(since $G$ is connected and distinct from an odd cycle).
Furthermore, two adjacent vertices of $C$ cannot both belong to $N(x)$
(since $G$ is $K_3$-free),
and therefore at most $\frac{|C|-1}{2}$ vertices belong to $C \cap N(x)$
(since $|C|$ is odd).
Consequently, we note that $G$ contains $\claw^+$ as an induced subgraph, a contradiction.
\end{proof}

\begin{lemma}\label{indep5}
Let $G$ be a connected $\{ K_1 \cup \claw, K_3 \}$-free graph.
If $G$ is of independence at least~$5$, then 
$G$ is either a path or a cycle or a complete bipartite graph.
In particular, there is an integer $n$ so that if $G$ has at least $n$
vertices, then it satisfies the condition.
For instance, we can choose $n = R(5, 3)$.
\end{lemma}

\begin{proof}
We consider a maximum independent set, say $I$, of vertices of $G$ (assuming that $|I| \geq 5$).
In addition, we can assume that $G$ is neither a path nor a cycle.
In particular, $G$ contains a vertex of degree at least $3$ (since $G$ is connected).
Clearly, every such vertex is the centre of an induced copy of $\claw$
(since $G$ is $K_3$-free).
Furthermore, we observe that $G$ contains such vertex, say $x$, not belonging to $I$
(since $G$ is $K_1 \cup \claw$-free and $|I| \geq 5$).
In the remainder of the proof,
we shall use the fact that every vertex adjacent to at least three vertices of $I$
is, in fact, adjacent to all vertices of $I$ (since $G$ is $K_1 \cup \claw$-free).

We show that $x$ is adjacent to at least three vertices of $I$ (and thus to all).
To the contrary, we suppose that $x$ is adjacent to at most two vertices of $I$.
Clearly, $x$ is adjacent to at least one vertex of $I$
(by the choice of $x$ and $I$).
Since $G$ is $K_1 \cup \claw$-free,
there is a vertex $x'$ adjacent to $x$ and to at least three vertices of~$I$.
Thus, $x'$ is adjacent to all vertices of~$I$,
which contradicts the fact that $G$ is $K_3$-free.

We consider a vertex $y$ not belonging to $I \cup \{ x \}$.
Clearly, $y$ is adjacent to a vertex of $I$,
and thus $y$ is not adjacent to~$x$
(since $G$ is $K_3$-free).
Consequently, we get that $y$ is adjacent to at least $|I| - 2$ vertices of $I$
(since $G$ is $K_1 \cup \claw$-free),
and thus $y$ is adjacent to all vertices of $I$.
It follows that $G$ is a complete bipartite graph.
\end{proof}

\begin{lemma}\label{applyRamsey}
For every positive integer $k$, there exists $n$ such that
every $\{ kK_1 \cup K_2, K_3 \}$-free graph
(on at least $n$ vertices) is bipartite.
For instance,
$n = R \left( k - 1, 3 \right) + \frac{8k(k^2-1)}{2k + 1} + 2k + 2$
will do.
\end{lemma}

\begin{proof}
We consider $n$ chosen as suggested and a $\{ kK_1 \cup K_2, K_3 \}$-free graph $G$
which is not bipartite, and we show that $G$ has at most $n-1$ vertices.
Since $G$ is not bipartite, there is an induced cycle $C_\ell$ where $\ell$ is odd.
Clearly, this cycle contains
$\frac{\ell - 3}{2} K_1 \cup K_2$ as an induced subgraph. 
Hence, we have $5 \leq \ell \leq 2k + 2$. 

We fix a set, say $C$, of vertices inducing $C_\ell$,
and let $O$ be the set of all vertices of $G$
which are adjacent to no vertex of $C$.
We note that the graph induced by $O$ is
$\{ (k - \frac{\ell - 3}{2})K_1, K_3 \}$-free.
By Theorem~\ref{Ramsey},
we get
\begin{equation*}
|O| \leq  R \left( k - \frac{\ell - 3}{2}, 3 \right) - 1 \leq  R \left( k - 1, 3 \right) - 1.
\end{equation*}

In addition, we let $N = V(G) \sm (C \cup O)$,
and we show that $|N| \leq \frac{8k(k^2-1)}{2k+1}$.
We consider a pair $\{ v, e \}$ such that
$v$ is a vertex of $C$
and $e$ is an edge of the cycle induced by $C$ and $v$ is not incident with $e$.
For each such pair,
we let $N_{v,\overline{e}}$ denote the set of all vertices of $N$
adjacent to $v$ and adjacent to none of the vertices incident with $e$.
We note that $N_{v,\overline{e}}$ is an independent set
(since $G$ is $K_3$-free),
and $|N_{v,\overline{e}}| \leq k - 1$
(since $G$ is $kK_1 \cup K_2$-free).
We observe that there are $\ell(\ell-2)$ distinct such pairs $\{ v, e \}$,
and thus $\ell(\ell-2)$ sets $N_{v,\overline{e}}$.  
Furthermore, we observe that every vertex of $N$
belongs to at least $\frac{\ell-1}{2}$ of these sets.
Consequently, we get
\begin{equation*}
|N| \leq \frac{2 \ell (\ell-2) (k-1)}{\ell-1} \leq \frac{8k(k^2-1)}{2k+1}.
\end{equation*}
We conclude that $G$ has  $|O| + |N| + \ell$ vertices,
and $|O| + |N| + \ell \leq n - 1$.
\end{proof}

\begin{obs}\label{cm}
Let $k$ and $\ell$ be positive integers
and $G$ be a $k K_1$-free graph (on at least $(k - 1)(\ell -1) + 1$ vertices)
whose every component is a complete multipartite graph.
Then $G$ contains $K_{\ell}$ as a subgraph.
\end{obs}

\begin{proof}
For the sake of a contradiction, we suppose that $G$ is $K_{\ell}$-free.
We let $G_1, \dots, G_p$ denote the graphs given by components of $G$
and $\alpha_1, \dots, \alpha_p$ denote their independence numbers.
We note that $\alpha_1 + \dots + \alpha_p \leq k - 1$,
and $G_i$ has at most $\alpha_i (\ell - 1)$ vertices for every $i = 1, \dots, p$.
Thus, $G$ has at most $(k - 1)(\ell -1)$ vertices, a contradiction.
\end{proof}

\begin{coro}\label{applyRamsey2}
Let $k$ be an integer greater than $2$
and let $n$ be an integer given by Lemma~\ref{applyRamsey} 
(with respect to $k$).
Let $G$ be a $\{ kK_1 \cup K_2, Z_1 \}$-free graph,
and $M$ be the subgraph of $G$ (possibly empty) given by all components of $G$
which are complete multipartite (possibly trivial),
and $N$ be the subgraph given by the remaining components.
Then at least one of the following statements is satisfied.
\begin{enumerate}
\item
$N$ is bipartite (possibly empty).
\item
$N$ has at most $n - 1$ vertices; and furthermore
if $M$ has at least $n(k - 2) - 2k + 5$ vertices,
then $M$ contains $K_{n - 1}$ as a subgraph. 
\end{enumerate}
\end{coro}

\begin{proof}
We recall that every component of $G$ is $K_3$-free or complete multipartite
(by Lemma~\ref{olariu}),
and hence $N$ is $K_3$-free.
We can assume that $N$ is not bipartite.
In particular, $N$ has at most $n - 1$ vertices (by Lemma~\ref{applyRamsey}).
Furthermore,
$N$ contains induced $K_1 \cup K_2$, and so
$M$ is $(k-1)K_1$-free.
We conclude that
if $M$ has at least $(k - 2)(n - 2) + 1$ vertices,
then it contains $K_{n - 1}$ 
(by Observation~\ref{cm}).
\end{proof}

In addition, we state the following characterisation
(a similar result was shown by Rao~\cite{Rao}).
The proof is given at the end of the present section.

\begin{lemma}
\label{l4}
Let $G$ be a graph containing $C_5$ as an induced subgraph.
Then $G$ is $\{ 2K_1 \cup K_2, \overline{K_1 \cup P_4} \}$-free
if and only if
it can be obtained from some of the graphs $G_1, \dots, G_{14}$
(depicted in Figure~\ref{figStructureOfG})
by blowing-up vertex $i$ to an independent set and
blowing-up vertices $c_1, \dots, c_5$ to complete graphs.
More precisely, the blow-up process is as follows.
For vertex $i$ (if present in the graph),
add any number of new vertices (possibly none)
adjacent precisely to the neighbours of $i$;
and then similarly for each $c_j$ in sequence (where $1 \leq j \leq 5$),
add in sequence any number of new vertices 
adjacent precisely to $c_j$ and to the neighbours of $c_j$.
\end{lemma}

\begin{figure}[h!]
    \centering
    \includegraphics[scale=0.7]{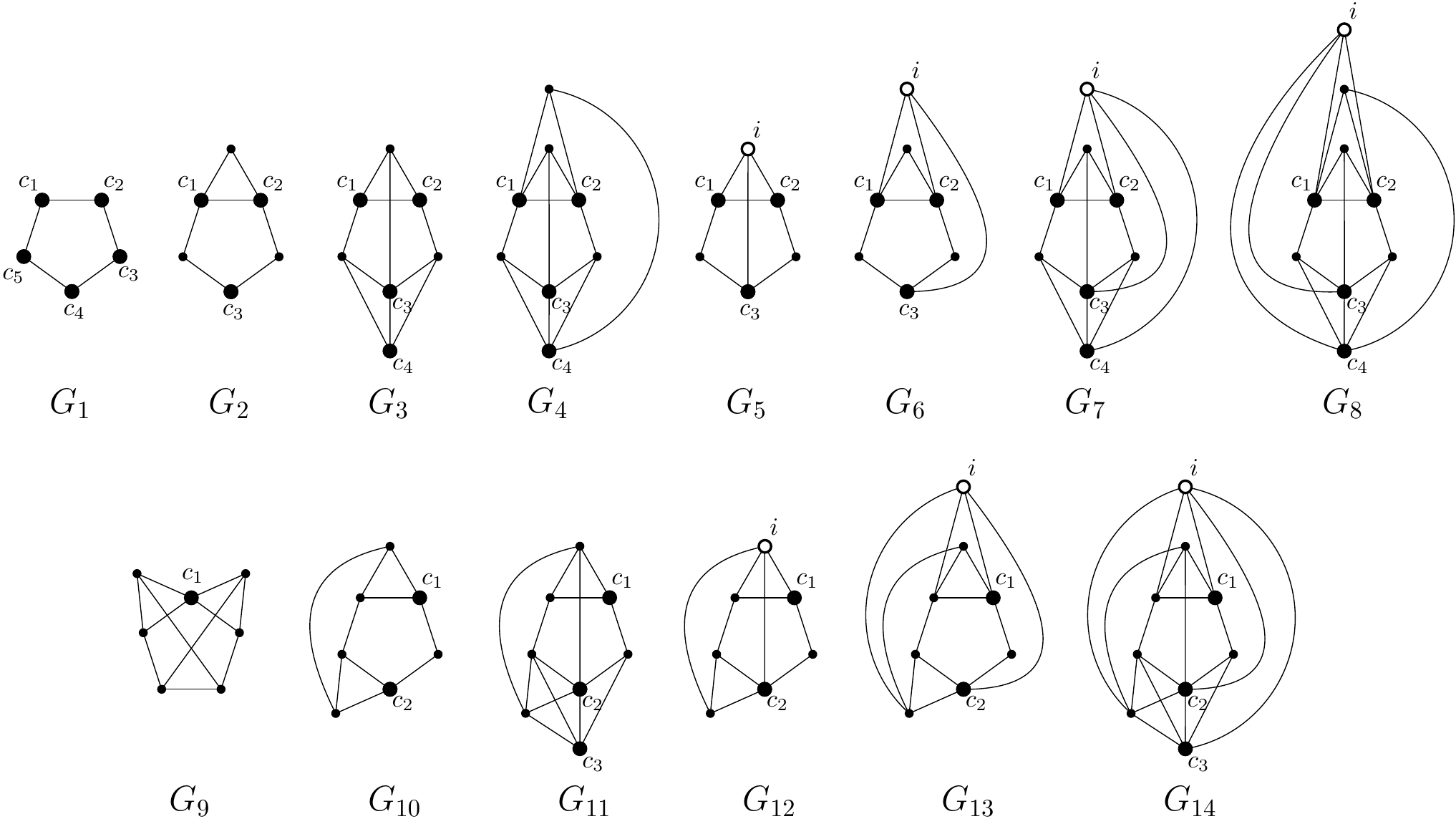}
    \caption{Graphs $G_1, \dots, G_{14}$
    whose particular vertices are labelled $i$ and $c_1, \dots, c_5$.
    We distinguish vertex $i$ by depicting it as a large 
    empty circle, and vertices $c_1, \dots, c_5$ by depicting them as large
    full circles.
	The graphs encode possible structure
	of a non-perfect $\{ 2K_1 \cup K_2, \overline{K_1 \cup P_4} \}$-free graph
	where each empty, full circle represents blowing-up a vertex 
	to an independent set, to a complete graph, respectively.}
    \label{figStructureOfG}
\end{figure}

We also show the following technical statement
on particular forbidden subgraphs of a forest
(it will be used for forests and for complements of forests).

\begin{lemma}\label{forest}
Let $F$ be a forest (on at least $4$ vertices)
distinct from $P_4$ and $\claw$.
The following statements are satisfied.
\begin{enumerate}
\item If $F$ is $\{ 4K_1, 2K_1 \cup K_2, 2K_2 \}$-free,
then $F$ is precisely the graph $K_1 \cup P_3$.
\item If $F$ is $\{ 2K_1 \cup P_3, 2K_2, K_{1,4} \}$-free,
then $F$ is either
$K_1 \cup P_3$ or $K_1 \cup P_4$ or $K_1 \cup \claw$ or $\claw^+$
or $(k + 2)K_1$ or $kK_1 \cup K_2$
(for some $k \geq 2$).
\end{enumerate}
\end{lemma}

\begin{proof}
In order to show statement (1), we observe the following facts.
\begin{itemize}
\item
The maximum degree of $F$ is at most $2$ 
(since $F$ is a $\{ 4K_1, 2K_1 \cup K_2 \}$-free forest distinct from $\claw$).
\item
$F$ has at most two components, 
(since it has at least $4$ vertices, and it is $\{ 4K_1, 2K_1 \cup K_2 \}$-free).
\item
If $F$ has two components, then one is trivial
(since it is $2K_2$-free).
\end{itemize}
Consequently, we get that all edges of $F$ belong to a common path
and that $F$ has at most two components.
Thus, $F$ contains induced $K_1 \cup P_3$
(since $F$ has at least $4$ vertices, and it is distinct from $P_4$).
Furthermore, $F$ cannot contain more vertices
(since it is $2K_1 \cup K_2$-free).

We show (2).
We can assume that $F$ has at least two edges
(otherwise $F$ is $(k + 2)K_1$ or $kK_1 \cup K_2$ for some $k \geq 2$,
and the statement is satisfied).
We note that $F$ has precisely one non-trivial component
(since it is $2K_2$-free); and we let $T$ denote the tree given by this component.

In addition, we can assume that $F$ is distinct from $K_1 \cup P_3$ and $K_1 \cup P_4$.
We observe that $T$ is not a path
(since $F$ has at least $4$ vertices
and it is distinct from $P_4$ and it is $\{ 2K_1 \cup P_3, 2K_2 \}$-free).
Hence, $T$ contains $\claw$, and so
$F$ contains induced $K_1 \cup \claw$ or $\claw^+$
(since $F$ is distinct from $\claw$ and it is $K_{1,4}$-free).
We conclude that $F$ is, in fact, $K_1 \cup \claw$ or $\claw^+$
(since it is $\{ 2K_1 \cup P_3, 2K_2, K_{1,4} \}$-free).
\end{proof}

%

We also state the following two facts on induced subgraphs of $\overline{K_3 \cup P_4}$,
and we give short proofs.
(This could also be shown simply by checking all graphs on at most $8$ vertices
with the help of a computer.)

\begin{obs}
\label{cK3P4on5}
There are precisely five distinct induced subgraphs of $\overline{K_3 \cup P_4}$ on $4$ vertices,
and precisely five such subgraphs on $5$ vertices.
Namely,
$\splitatcommas{
P_4, \claw, Z_1, D, C_4
}$
and
$\splitatcommas{
\overline{K_1 \cup P_4}, K_{1,2,2}, \overline{K_2 \cup P_3}, K_{1,1,3}, K_{2,3} 
}$.
\end{obs}

\begin{proof}
We discuss graphs obtained from $\overline{K_3 \cup P_4}$
by removing $3$ vertices.
We consider the maximum independent set of $\overline{K_3 \cup P_4}$
and let $i$ be the number of vertices of this set which are being removed.
For the case that $i = 3$, we note that the resulting graph is $P_4$.
For $i = 2$, we get $Z_1$ or $D$,
and for $i = 1$, we get $D$ or $C_4$.
Finally $i = 0$, gives $\claw$.
Similarly for removing $2$ vertices, we discuss the cases
and get the subgraphs on $5$ vertices.
\end{proof}

\begin{obs}
\label{cK3P4}
A graph is an induced subgraph of $\overline{K_3 \cup P_4}$
if and only if
it is
$\splitatcommas{
\{ 4K_1, 2K_1 \cup K_2, K_1 \cup P_3, 2K_2, K_1 \cup K_3, K_4, C_5,
\overline{P_5}, K_{3,3}, K_{2,2,2} \}
}$-free.
\end{obs}
\begin{proof}
We note that $\overline{K_3 \cup P_4}$ satisfies the property,
and hence it is satisfied by each of its induced subgraphs.

We consider a
$\splitatcommas{
\{ 4K_1, 2K_1 \cup K_2, K_1 \cup P_3, 2K_2, K_1 \cup K_3, K_4, C_5,
\overline{P_5}, K_{3,3}, K_{2,2,2} \}
}$-free graph $G$,
and we discuss two cases and show that 
it is an induced subgraph of $\overline{K_3 \cup P_4}$.

First, we suppose that $G$ contains an independent set, say $I$, of size $3$.
We note that every vertex of $V(G) \sm I$ is adjacent to all vertices of $I$
(since $G$ is $\{ 4K_1, 2K_1 \cup K_2, K_1 \cup P_3 \}$-free),
and that the graph $G - I$ is $\{ K_3, C_4, C_5, 3K_1, 2K_2 \}$-free
(since $G$ is $\{ K_4, K_{2,2,2}, C_5, K_{3,3}, 2K_2 \}$-free).
In particular, $G - I$ is a $\{ 3K_1, 2K_2 \}$-free forest,
and thus an induced subgraph of $P_4$ and the statement follows.

Next, we suppose that $G$ is $3K_1$-free, and we consider the complement of $G$.
Similarly, we note that $\overline{G}$ is a forest
(since $G$ is $\{ 3K_1, 2K_2, C_5, \overline{P_5} \}$-free).
Furthermore, $\overline{G}$ is $\{ 4K_1, 3K_2, \claw, P_5 \}$-free
(since $G$ is $\{ K_4, K_{2,2,2}, K_1 \cup K_3, \overline{P_5} \}$-free).
We observe that $\overline{G}$ is an induced subgraph of 
$K_2 \cup P_4$, and the statement follows.
\end{proof}

We conclude this section by proving Lemma~\ref{l4}.

\begin{proof}[Proof of Lemma~\ref{l4}]
\setcounter{claim}{0}
\setcounter{claimprefix}{\getrefnumber{l4}}
We consider the graphs obtained by the construction.
The `if part' of the statement follows by observing that they are
$\{ 2K_1 \cup K_2, \overline{K_1 \cup P_4} \}$-free.

We show the `only if part' of the statement.
We consider a set of vertices, say $C$, inducing $C_5$.
For every vertex $u$ of $V(G) \sm C$,
we get that $2 \leq |N(u) \cap C| \leq 3$
(since $G$ is $\{ 2K_1 \cup K_2, \overline{K_1 \cup P_4} \}$-free);
and we say that $u$ is \emph{blue} if $N(u) \cap C$ induces $P_3$,  
and $u$ is \emph{red} if $N(u) \cap C$ induces $K_2$ or $K_1 \cup K_2$
(see Figure~\ref{figBlueOrRed}).

\begin{figure}[h!]
    \centering
    \includegraphics[scale=0.7]{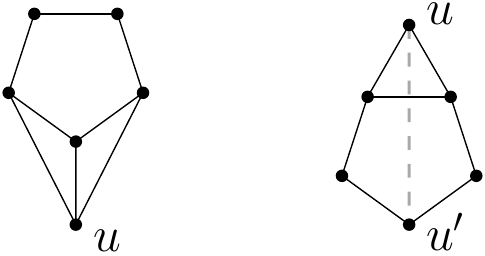}
    \caption{Possible graphs induced by $C \cup \{u\}$.
    Edge $uu'$ (depicted as dashed grey) may or may not be present in the graph;
    and so the picture on the right-hand side, in fact, represents two graphs.
    Depending on $N(u) \cap C$, vertex $u$ is blue (left-hand side) or red (right-hand side).}
    \label{figBlueOrRed}
\end{figure}

In addition, we consider a set, say $A$, consisting of two vertices of $V(G) \sm C$,
and we shall discuss the graph induced by $C \cup A$.
We let $X = C \cup A$, and let
$r$ be the number of red vertices in $A$,
and $H_1, \dots, H_8$ be the graphs depicted in Figure~\ref{figC++}
(considered with the dashed grey edges).
We show three claims.

\begin{figure}[h!]
    \centering
    \includegraphics[scale=0.7]{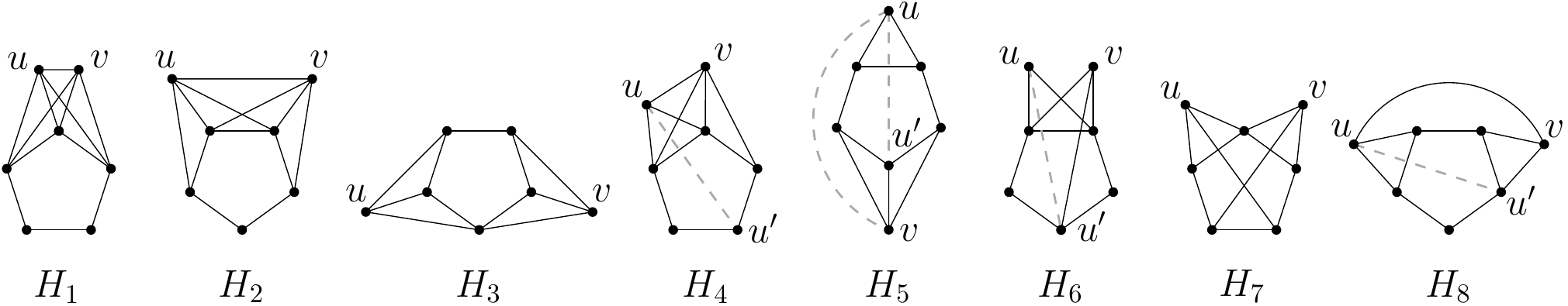}
    \caption{Graphs $H_1, \dots, H_8$.
    The vertices of $A$ are labelled $u$ and $v$,
    and one particular vertex of $C$ is labelled $u'$ in graphs $H_4, H_5, H_6$ and $H_8$.
    We distinguish edge $uv$ of $H_5$ and all edges $uu'$ 
    by depicting them as dashed grey.}
    \label{figC++}
\end{figure}

\begin{claim} 
\label{ci}
If $r = 0$, then $X$ induces either $H_1$ or $H_2$ or $H_3$.
\end{claim}
\begin{proofcl}[Proof of Claim~\ref{ci}]
For the sake of a contradiction,
we suppose that $X$ induces none of $H_1, H_2, H_3$.
Considering the number of vertices of $C$ which are adjacent to both vertices of $A$,
we discuss three cases, see Figure~\ref{figForbiddenC++} (first row),
and we observe that the graph induced by
$X$ is not 
$\{ 2K_1 \cup K_2, \overline{K_1 \cup P_4} \}$-free,
a contradiction.
\end{proofcl}

\begin{claim} 
\label{cii}
If $r = 1$, then the graph induced by $X$
can be obtained from either $H_4$ or $H_5$ by removing some of the dashed grey edges (possibly none).
\end{claim}
\begin{proofcl}[Proof of Claim~\ref{cii}]
For the sake of a contradiction,
we suppose that the graph induced by $X$
violates the claim.
We discuss all cases,
see Figure~\ref{figForbiddenC++} (second row),
and we observe that the graph is not 
$\{ 2K_1 \cup K_2, \overline{K_1 \cup P_4} \}$-free,
a contradiction.
\end{proofcl}

\begin{figure}[h!]
    \centering
    \includegraphics[scale=0.7]{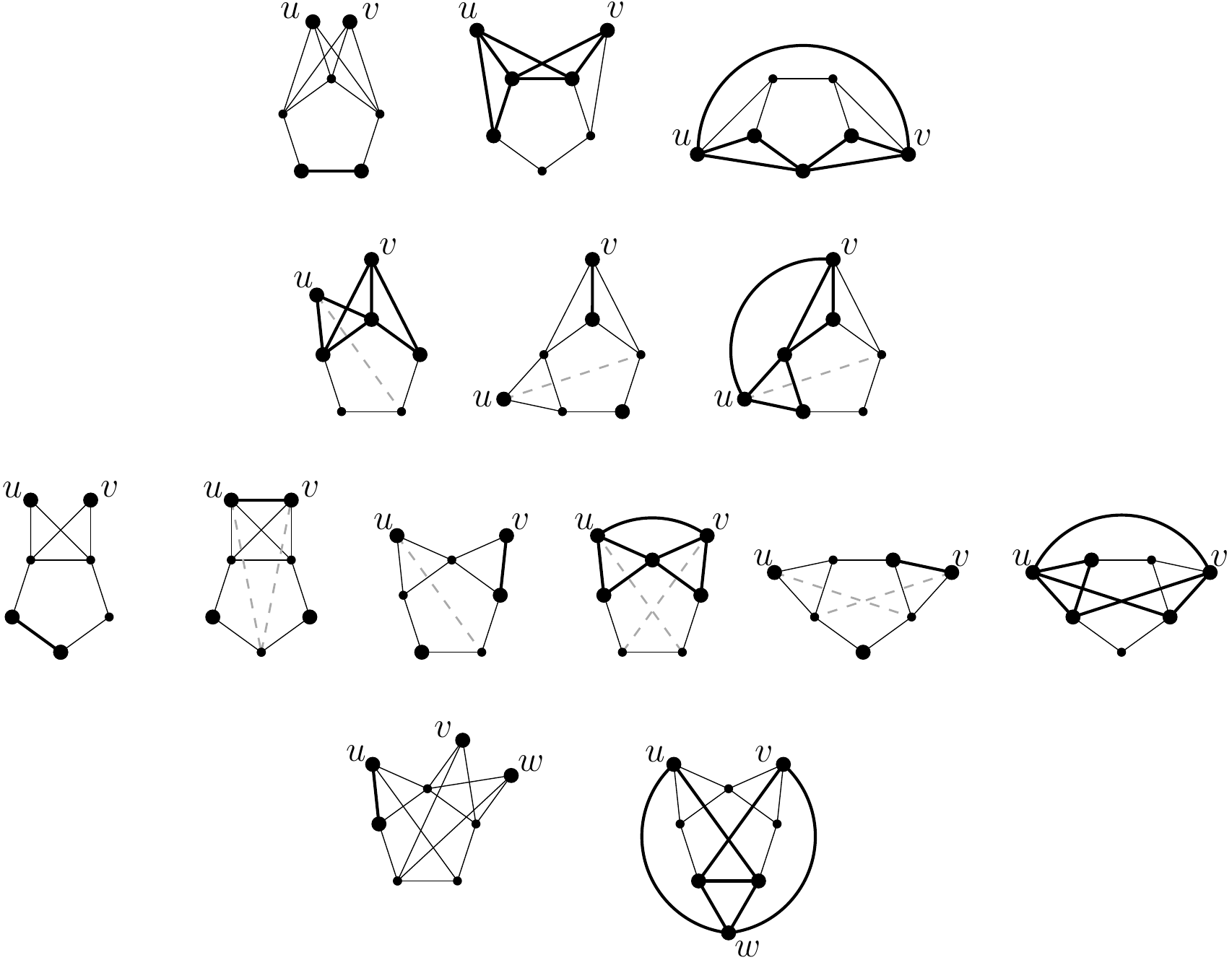}
    \caption{Particular adjacencies among the vertices of $A$ (labelled $u$ and $v$)
    and~$C$, and among $A$ and $C'$ (third row),
    and among $A$ and $C'$ and vertex $w$ (fourth row).
    The dashed grey edges indicate that the vertices may or may not be adjacent;
    for instance, the first picture in the second row represents two graphs
    (both containing induced $\overline{K_1 \cup P_4}$).
    In each picture, induced $2K_1 \cup K_2$ or $\overline{K_1 \cup P_4}$ is
    highlighted (the vertices and edges are depicted as bold).}
    \label{figForbiddenC++}
\end{figure}

We use the similarity of the adjacencies of blue vertices and vertices of
$C$ given by Claims~\ref{ci} and~\ref{cii}, and we consider
a set, say~$C'$, of five non-red vertices inducing $C_5$
(so that $C'$ and $A$ are disjoint).
We let $X' = C' \cup A$ and consider the graph induced by $X'$.

\begin{claim} 
\label{ciii}
If $r = 2$, then the graph induced by $X'$ is either $H_7$ or
it can be obtained from $H_6$ or $H_8$ by removing some of the dashed grey edges (possibly none).
Furthermore, if $X'$ induces $H_7$, then $G$ has precisely two red vertices.
\end{claim}

\begin{proofcl}[Proof of Claim~\ref{ciii}]
Similarly as above, we suppose that the graph induced by $X'$
violates the first statement of the claim,
and we observe that it is not 
$\{ 2K_1 \cup K_2, \overline{K_1 \cup P_4} \}$-free,
see Figure~\ref{figForbiddenC++} (third row).

The second statement of the claim is also shown by contradiction.
We suppose that $X'$ induces $H_7$ and that there is another red vertex, say $w$,
and we discuss the graph induced by $X' \cup \{w\}$.
We apply the first statement of the claim to
all $2$-element subsets of $A \cup \{w\}$ and
observe how this reduces the number of cases.
We discuss the remaining cases, see Figure~\ref{figForbiddenC++} (fourth row),
and we conclude that the graph induced by
$X' \cup \{w\}$ is not $\{ 2K_1 \cup K_2, \overline{K_1 \cup P_4} \}$-free,
a contradiction.
\end{proofcl}

In addition, we consider an edge $e$ whose both ends belong to $C$
and the set, say $R_e$, of all red vertices adjacent to both ends of $e$;
and we choose $e$ so that $|R_e|$ is maximised.
Using Claim~\ref{ciii}, we observe that
at most one red vertex does not belong to $R_e$.
We let $o$ be the vertex of $C$ which is adjacent to no end of $e$,
and we let $O$ be the set consisting of vertex $o$ and all vertices added for $o$
by the blow-up process.
We show the following.

\begin{claim} 
\label{civ}
At most two vertices of $R_e$ have the property that they are not adjacent to all vertices of $O$.
Furthermore, if there are two such vertices, then every red vertex belongs to $R_e$
and every vertex of $O$ is adjacent to precisely one of the two.
\end{claim}

\begin{proofcl}[Proof of Claim~\ref{civ}]
We note that the first statement of the claim follows from the second statement.
Hence, it is sufficient to show the second statement,
and proceed by contradiction.
We consider two distinct vertices $u,v$ of $R_e$, and two vertices $u',v'$ of $O$
such that $u$ is not adjacent to $u'$ and $v$ is not adjacent to $v'$
(and note that $u'$ and $v'$ are distinct by Claim~\ref{ciii}).
For the sake of a contradiction, we suppose that there is a vertex $x$
such that either $x$ is red and not belonging to $R_e$,
or $x$ belongs to $O$ and it is not true that $x$ is adjacent to precisely one of $u,v$
(and thus it is adjacent to both by Claim~\ref{ciii}).
Using Claims~\ref{ci}, \ref{cii} and~\ref{ciii},
we note that $\{u, v, u',v'\}$ induces $P_4$ and 
that $x$ is adjacent to $u,v, u'$ and $v'$, a contradiction
(see Figure~\ref{figDisjointUnion}).
\end{proofcl}

\begin{figure}[h!]
    \centering
    \includegraphics[scale=0.7]{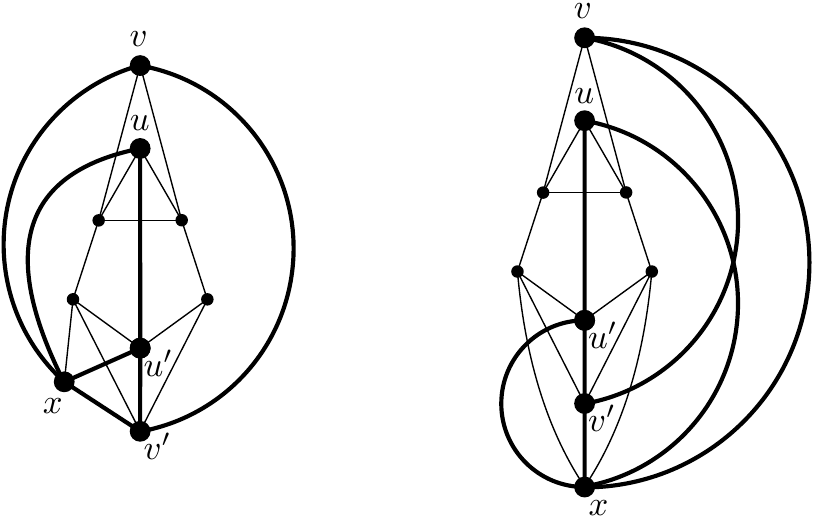}
    \caption{Graphs induced by $C \cup \{u, v, u',v', x\}$.
    In both pictures, induced $\overline{K_1 \cup P_4}$ is highlighted.}
    \label{figDisjointUnion}
\end{figure}

We apply Claims~\ref{ci}, \ref{cii} and~\ref{ciii}
to all $2$-element subsets of $V(G) \sm C$,
and we discuss what we know about the structure of $G$.
We recall that Claim~\ref{ci} gives the adjacencies between all non-red vertices.
In particular, if $G$ has no red vertex, then it can be obtained from $G_1$
by the blow-up process (described in Lemma~\ref{l4}); and so we can assume that $G$ has a red vertex.
We note that Claim~\ref{cii} restricts the relative positions of red and blue vertices (viewed from $C$).
Finally, Claim~\ref{ciii} gives the adjacencies between all red vertices,
and also some of the adjacencies between red and non-red vertices.
It remains to discuss the adjacencies
corresponding to dashed grey edges of graphs $H_6$ and $H_8$,
see Figure~\ref{figC++}.

Since at most one red vertex does not belong to $R_e$,
most of the remaining adjacencies are given by Claim~\ref{civ}.
We discuss two cases.
For the case that every red vertex belongs to $R_e$,
we observe that $G$ can be obtained from some of
$G_2, \dots, G_8$ by the blow-up process.
Otherwise,
we similarly conclude that $G$ can be obtained from some of $G_9, \dots, G_{14}$.
\end{proof}

\section{Showing perfectness}
\label{sPerf}
In the present section,
we collect sufficient conditions for a graph to be perfect
and we show the following. 
\begin{prop}
\label{perf}
Let $G$ be a graph.
If some of the following conditions is satisfied,
then $G$ is perfect.
\begin{enumerate}
\item
$G$ is a $\{ 2K_1 \cup K_2, D \}$-free graph
distinct from $C_5$ and from the graphs $E_1, \dots, E_4$ depicted in Figure~\ref{figE}.
\item
$G$ is a $\{ K_1 \cup P_3, \overline{K_3 \cup P_4} \}$-free graph
of independence at least $3$.
\item
$G$ is a connected $\{ \claw^+, Z_1 \}$-free graph
distinct from an odd cycle.
\item
$G$ is a connected $\{ K_1 \cup \claw, Z_1 \}$-free graph
distinct from an odd cycle
and $G$ has at least $n$ vertices
where $n$ is given by Lemma~\ref{indep5}.
\item
$G$ is a connected $\{ kK_1 \cup K_2, Z_1\}$-free graph (where $k \geq 3$)
on at least $n$ vertices where $n$ is given by Lemma~\ref{applyRamsey}.
\item
$G$ is a $\{ kK_1 \cup K_2, K_3 \}$-free graph (where $k \geq 3$)
on at least $n$ vertices where $n$ is given by Lemma~\ref{applyRamsey}.
\item
$G$ is a $\{ 3K_1, \overline{kK_1 \cup K_2} \}$-free graph (where $k \geq 3$)
on at least $n$ vertices where $n$ is given by Lemma~\ref{applyRamsey}.
\item
$G$ is a $\{ 2K_1 \cup K_2, Z_1\}$-free graph
distinct from $C_5$.
\item
$G$ is a $\{ K_1 \cup P_3, Z_1\}$-free graph
distinct from $C_5$.
\item
$G$ is a $\{ K_1 \cup P_3, D \}$-free graph
distinct from $C_5$.
%
\end{enumerate}
\end{prop}

\begin{figure}[h!]
    \centering
    \includegraphics[scale=0.7]{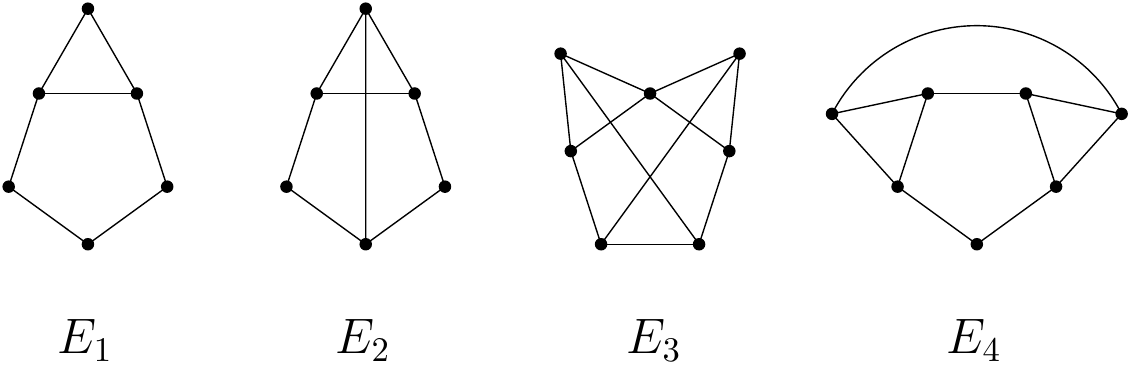}
    \caption{Graphs $E_1, \ldots, E_4$.
    We note that $E_1$ is, in fact, the complement of $E_2$,
    and $E_3$ is the complement of $E_4$.
    (The graphs are also included in Figure~\ref{figStructureOfG}.
    Namely, $E_1, E_2, E_3, E_4$ is isomorphic to $G_2, G_5, G_9, G_{10}$, respectively.)}
    \label{figE}
\end{figure}

For conditions (3), (4) \dots, (10)
we argue using the structural statements shown in Section~\ref{sForb},
and for (1) and (2) using the following result shown in~\cite{CRST}.

\begin{theorem}[Strong Perfect Graph Theorem]
\label{tSPGT}
A graph is perfect if and only if neither the graph nor its complement
contain an induced cycle whose length is odd and at least $5$.
\end{theorem}

\begin{proof}[Proof of Proposition~\ref{perf}]
We let $\CC$ denote the family of all cycles whose length is odd and at least $7$,
and $\overline{\CC}$ denote the family of all graphs whose complement belongs to $\CC$.
We show that each of the conditions (1), (2), \dots, (10)
implies that $G$ is perfect.

First, we suppose that condition (1) is satisfied.
For the sake of a contradiction, we suppose that $G$ is not perfect.
By Theorem~\ref{tSPGT},
$G$ contains a graph from
$\{C_5\} \cup \CC \cup \overline{\CC}$ as an induced subgraph.
Furthermore, we note that $G$ is $\CC \cup \overline{\CC}$-free
(since no graph of $\CC$ is $2K_1 \cup K_2$-free
and no graph of $\overline{\CC}$ is $D$-free),
and hence, $G$ contains $C_5$ as an induced subgraph.
(Now, the statement can be deduced using Lemma~\ref{l4}.
For the sake of clarity, we give a short proof not using the lemma.)
We consider a set $C$ of vertices inducing $C_5$ in $G$.
We note that every vertex of $G$ is adjacent to a vertex of $C$
(since $G$ is $2K_1 \cup K_2$-free),
and that $G$ contains a vertex not belonging to $C$
(since $G$ is distinct from $C_5$).
Considering such vertex $x$,
we observe that the graph induced by $C \cup \{x\}$
is either $E_1$ or $E_2$ (since $G$ is $\{ 2K_1 \cup K_2, D \}$-free).

In addition, we consider two vertices, say $x$ and $x'$, not belonging to $C$,
and we discuss the graph induced by $C \cup \{x, x'\}$
(see Figure~\ref{figCcupxx}).
We observe that $C \cup \{x, x'\}$ induces either $E_3$ or $E_4$.
\begin{figure}[h!]
    \centering
    \includegraphics[scale=0.7]{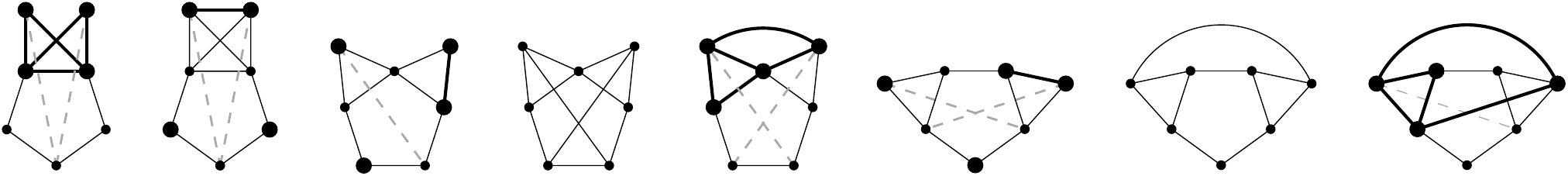}
    \caption{Possible ways of connecting $x$ and $x'$ to $C$
    (the edges depicted as dashed grey may or may not be present in the graph).
    Some of the resulting graphs contain induced $2K_1 \cup K_2$ or induced $D$
    (highlighted in the picture by depicting the vertices and edges as bold).}
    \label{figCcupxx}
\end{figure}
Clearly, this reasoning applies to every pair of vertices not belonging to $C$,
and thus $G$ has at most $7$ vertices.
We conclude that $G$ is one of the graphs $E_1, \dots, E_4$, a contradiction.

We suppose that (2) is satisfied.
For the sake of a contradiction, we suppose that
$G$ is not perfect, that is, contains a graph from
$\{C_5\} \cup \CC \cup \overline{\CC}$ as an induced subgraph (by Theorem~\ref{tSPGT}).
Clearly, $G$ is $\CC$-free (since it is $K_1 \cup P_3$-free).
Hence, $G$ contains a graph from $\{C_5\} \cup \overline{\CC}$
as an induced subgraph;
and let $A$ be a set of vertices inducing such subgraph.
We note that every vertex of $G$ is adjacent to a vertex of~$A$
(since $G$ is $K_1 \cup P_3$-free).

In addition, we consider a set $I$ of three independent vertices of $G$,
and a pair $N$ of non-adjacent vertices of $A$.
For every vertex $x$ of $I \sm N$,
we observe that $x$ is adjacent to at least one vertex of $N$
(otherwise,
there would be a pair $N'$ of non-adjacent vertices of $A \sm N(x)$ so that
there is a vertex of $N(x)$ which is adjacent to precisely one vertex of $N'$,
contradicting the assumption that $G$ is $K_1 \cup P_3$-free).
Consequently, we note that no vertex of $N$ belongs to $I$
(since $G$ is $K_1 \cup P_3$-free).
In particular, a vertex of $N$ is adjacent
to at least two vertices of $I$,
and hence to all vertices of $I$ (since $G$ is $K_1 \cup P_3$-free).
It follows that the other vertex of $N$ is
also adjacent to all vertices of $I$. 
We recall that the choice of $N$ was arbitrary,
and thus every vertex of $A$ is, in fact, adjacent to all vertices of $I$.
We consider a subset, say $P$, of $A$ inducing $P_4$,
and we conclude that $I \cup P$ induces $\overline{K_3 \cup P_4}$,
a contradiction. 

We suppose that (3) or (4) or (5) is satisfied.
In either case, $G$ is $K_3$-free or complete multipartite
by Lemma~\ref{olariu}.
Furthermore if $G$ is $K_3$-free, then we get that it is bipartite
(we apply Observation~\ref{bip} or Lemma~\ref{indep5}
or Lemma~\ref{applyRamsey}, respectively).
Clearly, every bipartite or complete multipartite graph is perfect.

Similarly, for the case that (6) is satisfied,
we get that $G$ is bipartite (by Lemma~\ref{indep5}),
and thus perfect.

We suppose that (7) is satisfied.
We note that the complement of $G$ is $\{ kK_1 \cup K_2, K_3 \}$-free,
and so it satisfies condition (6).
We recall that a graph is perfect if and only if its complement is perfect
(by~\cite{L} or by Theorem~\ref{tSPGT}), and the perfectness of $G$ follows.

We suppose that (8) is satisfied.
If $G$ is connected, then we note that condition (3) is satisfied
and the statement follows.
Otherwise,
we get that every component of $G$ is $K_1 \cup K_2$-free.
In particular, $G$ is $P_4$-free, and we recall that this implies perfectness
(for instance, by~\cite{Seinsche} or by Theorem~\ref{tSPGT}).
A similar argument applies for the case that condition (8) is satisfied.

Finally, if (9) is satisfied, then
we note that the complement of $G$ satisfies (8),
and the statement follows.
%
%
%
%
\end{proof}
\section{Showing $\omega$-colourability}
\label{sOmega}

Considering sufficient conditions for $\omega$-colourability,
we show the following.
\begin{prop}
\label{p1}
Every $\{ 2K_1 \cup K_2, \overline{K_1 \cup P_4} \}$-free graph
of independence at least~$3$
is $\omega$-colourable.
\end{prop}

\begin{prop}
\label{omega}
Let $k, \ell$ and $n$ be integers so that $k \geq 3$ and $\ell \geq 2$,
and $G$ be a graph on at least $n$ vertices.
If some of the following conditions is satisfied,
then $G$ is $\omega$-colourable.
\begin{enumerate}
\item
$G$ is $\{ kK_1 \cup K_2, Z_1 \}$-free and $n$ is sufficiently large.
For instance, we can choose
$n = (k - 1) \left( R \left( k - 1, 3 \right) + \frac{8k(k^2-1)}{2k + 1} + 2k \right) + 2$.
\item
$G$ is $\{ kK_1 \cup K_2, D \}$-free and $n$ is sufficiently large.
For instance, $n = R(2k, R(k, k) + k)$ will do.
\item
$G$ is $\{ (k+1)K_1, \overline{\ell K_1 \cup K_2} \}$-free and $n$ is sufficiently large.
For instance, $n = R(k+1, R(k, k(\ell - 1)) + k(\ell - 1))$ will do.
\end{enumerate}
\end{prop}

The proofs of Propositions~\ref{p1} and~\ref{omega} are given below.
For Proposition~\ref{p1}, the proof follows easily
by using Theorem~\ref{tSPGT}
and the `only if part' of Lemma~\ref{l4}
(a similar reasoning and deeper results can be found in
the paper of Karthick and Maffray~\cite{KM}).
For each item of Proposition~\ref{omega},
we examine the structure of a major part of a graph and colour it,
and we show that the remaining part is small and the colouring extends.
For condition (1), the statement follows from Corollary~\ref{applyRamsey2}.
For conditions (2) and (3), we repeatedly apply Theorem~\ref{Ramsey}.
The core of the proof is to show the following lemma on a slightly more general class of graphs.

\begin{lemma}
\label{l5}
Let $k \geq 3$ and $\ell \geq 2$
and $G$ be a $\{ kK_1 \cup K_2, \overline{\ell K_1 \cup K_2} \}$-free graph
and $\omega$ be its clique number.
If $\omega$ is sufficiently large,
then $G$ is $\omega$-colourable.
For instance, $\omega \geq R(k, k(\ell - 1)) + k(\ell - 1)$ will do.
\end{lemma}

We start by showing Proposition~\ref{p1}.

\begin{proof}[Proof of Proposition~\ref{p1}]
We let $G$ be a $\{ 2K_1 \cup K_2, \overline{K_1 \cup P_4} \}$-free graph 
of independence at least~$3$.
Clearly, we can assume that $G$ is not perfect.
Combining Theorem~\ref{tSPGT}
and the fact that $G$ is $\{ 2K_1 \cup K_2, \overline{K_1 \cup P_4} \}$-free,
we note that $G$ contains $C_5$ as an induced subgraph;
and we consider a set, say $C$, of vertices inducing it.

We recall the blow-up process described in Lemma~\ref{l4}
and note that $G$ can be obtained 
from some of the graphs $G_2, \dots, G_{14}$
(since $G$ is of independence at least $3$).
To reduce the number of cases, we observe that $G$ can be extended by adding edges
so that the resulting graph
can, in fact, be obtained from $G_5, G_9$ or $G_{12}$
(recalling the notation used in the proof of Lemma~\ref{l4},
we add edges so that every vertex of $R_e$ is adjacent to all vertices of $O$
for a particular choice of $e$).
We let $G^+$ denote the resulting extended graph,
and we observe that $\omega(G) = \omega(G^+)$.
Hence, it is sufficient to
find an $\omega$-colouring of $G^+$.

\begin{figure}[h!]
    \centering
    \includegraphics[scale=0.7]{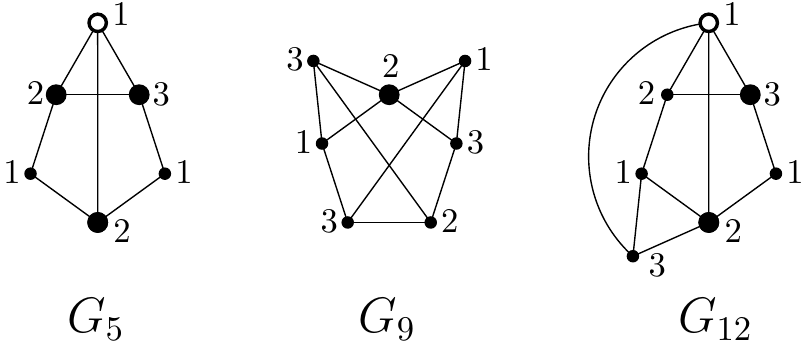}
    \caption{Possible structure of graph $G^+$ given by graphs $G_5, G_9, G_{12}$, 
    and particular $3$-colourings of these graph.}
    \label{figColouring}
\end{figure}

To this end, we consider a set of vertices of $G^+$
inducing $G_5, G_9$ or $G_{12}$
(a subgraph giving the structure of $G$).
We colour the vertices as indicated in Figure~\ref{figColouring},
and we extend this to a proper colouring of $G^+$ as follows.
We recall that vertex $i$ is blown-up to an independent set
and each of vertices $c_1, c_2$ and $c_3$ is blown-up to a complete graph,
and we refer to the new vertices as clones of $i, c_1, c_2$ and $c_3$, respectively. 
We colour the clones of $i$ using colour $1$,
and the clones of $c_1$ and $c_2$ using colours $4, \dots, \omega$,
and the clones of $c_3$ (for the case of graph $G_5$) using colours $3, \dots, \omega$.
We conclude that this yields an $\omega$-colouring of $G^+$.
\end{proof}

We show Lemma~\ref{l5}.

\begin{proof}[Proof of Lemma~\ref{l5}]
\setcounter{claim}{0}
\setcounter{claimprefix}{\getrefnumber{l5}}
We let $m = k(\ell - 1)$ and we suppose that $\omega \geq R(k, m) + m$.
We consider a set of vertices, say $V_1$,
inducing a maximum complete subgraph of $G$
(clearly, $|V_1| = \omega$),
and we check whether the graph $G - V_1$ contains $K_m$ as a subgraph.
If it does, then we consider a set of vertices, say $V_2$,
inducing a maximum complete subgraph in $G - V_1$
and we proceed by checking the graph $G - V_1 - V_2$ for $K_m$.
We continue this process until we obtain sets $V_1, \dots, V_p$
so that the graph $G - V_1 - \dots - V_p$ is $K_m$-free
(and $|V_p| \geq m$).

We show two claims.
\begin{claim} 
\label{ca}
Let $i$ be an integer satisfying $1 \leq i \leq p$.
Then every vertex of $V(G) \sm (V_1 \cup \dots \cup V_i)$ is adjacent to at most $\ell - 1$ vertices of $V_i$.
Furthermore, we get that $p \leq k$ and that the graph $G - V_1 - \dots - V_p$ is $(k-p+1)K_1$-free.
\end{claim}

\begin{proofcl}[Proof of Claim~\ref{ca}]
We consider a vertex $v$ of $V(G) \sm (V_1 \cup \dots \cup V_i)$,
and we note that at least one vertex of $V_i$ is not adjacent to $v$
(by the maximality of $V_i$).
Thus, $v$ is adjacent to at most $\ell - 1$ vertices of $V_i$
(since $G$ is $\overline{\ell K_1 \cup K_2}$-free).

We prove the second part of the claim by contradiction.
To this end,
we consider a maximum independent set (possibly empty), say $I$, of $G - V_1 - \dots - V_p$
and suppose that $p > k$ or $|I| > k-p$.
We will repeatedly use the first statement of the claim.
First, we show that there is an independent set of size $k$ in $G - V_1$
(this is clearly satisfied if $|I| \geq k$).
We suppose that $|I| < k$, and we let $j = k - |I|$.
In particular, we have $j < p$ (since $p > k$ or $|I| > k-p$).
We note that
$I$ can be extended to a larger independent set by adding (in sequence)
one vertex from each of $V_{j+1}, V_j, \dots, V_{2}$
(since $j < p$), and the resulting independent set is of size $k$.
Now, we use that there is an independent set of size $k$ in $G - V_1$,
and we note that it can be extended by adding two vertices of $V_1$
so that the resulting set induces $kK_1 \cup K_2$, a contradiction.
\end{proofcl}

\begin{claim} 
\label{cb}
Let $i$ be an integer satisfying $2 \leq i \leq p$.
If $X$ is a subset of $V_1 \cup \dots \cup V_{i-1}$
such that each vertex of $X$ has at least $\ell$ neighbours in $V_i$,
then $X \cup V_i$ induces a complete graph.
In particular, we have $\omega - |X| \geq |V_i|$.
\end{claim}
\begin{proofcl}[Proof of Claim~\ref{cb}]
Since $G$ is $\overline{\ell K_1 \cup K_2}$-free,
we get that each vertex of $X$ 
is adjacent to all vertices of $V_i$,
and consequently that all vertices of $X$ are adjacent.
Thus, $X \cup V_i$ induces a complete graph
and the inequality follows.
\end{proofcl}

Finally, we show that $G$ is $\omega$-colourable.
We start by colouring the subgraph induced by $V_1 \cup \dots \cup V_p$.
Clearly, the complete subgraph given by $V_1$ is $\omega$-colourable.
We suppose that there is an $\omega$-colouring of the subgraph induced by $V_1 \cup \dots \cup V_{i-1}$,
and we show that it can be extended to $V_1 \cup \dots \cup V_i$.
To~this end, we construct an auxiliary bipartite graph $(A, B)$ such that
the vertices of $A$ encode the vertices of $V_i$ and
the vertices of $B$ encode the colours,
and two vertices are adjacent if and only if the corresponding member of $V_i$
has no neighbour in $V_1 \cup \dots \cup V_{i-1}$ coloured by the corresponding colour.
We use Claims~\ref{ca} and~\ref{cb}
and show that set $A$ satisfies Hall's condition
(that is, $|N(S)| \geq |S|$ for every subset $S$ of $A$). 
Claim~\ref{ca} implies that every vertex of $V_i$
has at most $\ell - 1$ neighbours in each of $V_1, \dots, V_{i-1}$, and hence
every vertex of $A$ has degree at least $\omega - (i-1)(\ell - 1)$.
In particular, we get $|N(S)| \geq \omega - (i-1)(\ell - 1)$
for every subset $S$ of $A$,
and thus $|N(S)| \geq (i-1)(\ell - 1)$
(since  $\omega \geq 2k(\ell - 1)$ and $k \geq i$).
On the other hand, 
every vertex of $B \sm N(S)$ encodes a colour such that
each vertex of $V_i$ belonging to $S$ is adjacent to a vertex of
$V_1 \cup \dots \cup V_{i-1}$ coloured by this colour.
In particular if $|S| > (i-1)(\ell - 1)$, then
one of these vertices has at least $\ell$ neighbours in $V_i$
(since there are at most $i-1$ vertices of the same colour in
$V_1 \cup \dots \cup V_{i-1}$).
We note that the inequality from Claim~\ref{cb} translates to $|N(S)| \geq |A|$
(since $|A| = |V_i|$ and $|B| = \omega$).
Thus, we can apply Hall's theorem and obtain a matching covering~$A$,
and we note that the matching translates back to the desired extension of the colouring.

Hence, there is an $\omega$-colouring of the subgraph induced by $V_1 \cup \dots \cup V_p$,
and we extend it to $G$ as follows.
We recall that the graph $G - V_1 - \dots - V_p$ is $\{kK_1, K_m\}$-free,
and so it has fewer than $R(k,m)$ vertices (by Theorem~\ref{Ramsey}).
By Claim~\ref{ca}, each of these vertices has at most $p(\ell - 1)$ neighbours in $V_1 \cup \dots \cup V_p$,
and thus they can be readily coloured
(since $\omega \geq R(k,m) + m \geq R(k,m) + p(\ell - 1)$).
\end{proof}

Finally, we show Proposition~\ref{omega}.

\begin{proof}[Proof of Proposition~\ref{omega}]
We suppose that condition (1) is satisfied, and we choose $n$ as suggested.
We apply Corollary~\ref{applyRamsey2}
and we view $G$ as a disjoint union of graphs $M$ and $N$ (defined in Corollary~\ref{applyRamsey2}).
Clearly, $M$ is $\omega$-colourable,
and we get that $N$ is bipartite or small.
For the latter case, we note that
$N$ has at most $n' - 1$ vertices where
$$n' = R \left( k - 1, 3 \right) + \frac{8k(k^2-1)}{2k + 1} + 2k + 2,$$
and so $M$ has at least $n'(k - 2) - 2k + 5$ vertices
(since $n = n'(k - 1) - 2k + 4$),
and hence $M$ contains a complete subgraph of order $n' - 1$.
We conclude that the disjoint union of $M$ and $N$ is $\omega$-colourable.

We suppose that (2) is satisfied.
We choose $n$ as suggested,
and we consider a $\{ kK_1 \cup K_2, D \}$-free graph $G$ on at least $n$ vertices.
By Theorem~\ref{Ramsey},
$G$ contains an independent set of size $2k$
or a complete subgraph of order $R(k, k) + k$,
and we discuss these two cases.

First, we consider a maximum independent set, say $I$,
and suppose that $|I| \geq 2k$.
Clearly, every vertex of $V(G) \sm I$
is adjacent to a vertex of $I$ (by definition),
and so it is adjacent to more than $|I| - k$ vertices of $I$
(since $G$ is $kK_1 \cup K_2$-free).
Consequently, if $u$ and $v$ are vertices of $V(G) \sm I$,
then $u$ and $v$ have at least two common neighbours in $I$;
and hence $u$ and $v$ are non-adjacent
(since $G$ is $D$-free).
It follows that $G$ is bipartite, and thus $\omega$-colourable.

Next, we suppose that $\omega \geq R(k, k) + k$.
We note that graph $D$ can be viewed as $\overline{\ell K_1 \cup K_2}$ for $\ell = 2$,
and the $\omega$-colourability follows by Lemma~\ref{l5}.

Lastly, we suppose that (3) is satisfied.
By the choice of $n$ and by Theorem~\ref{Ramsey},
we get $\omega \geq R(k, k(\ell - 1)) + k(\ell - 1)$.
Thus, the statement follows by Lemma~\ref{l5}.
\end{proof}

\section{Discussing all remaining pairs}
\label{sFam}
In order to prove the `only if part' of 
Theorems~\ref{mainNoExceptions} and~\ref{mainFiniteExceptions},
we exclude all possible remaining pairs of forbidden subgraphs.
To this end, we consider families $\FF_1, \dots, \FF_{13}$ of graphs
depicted in Figures~\ref{figFamilies123}, \ref{figFamilies4to10} and~\ref{figFamilies11to13}.
In addition, we use Lemma~\ref{forest} and Observations~\ref{cK3P4on5} and~\ref{cK3P4},
and Theorem~\ref{alpha3} which resolves the pairs containing $\claw$.
As the main result of this section, we show the following.
\begin{prop}
\label{notOmega}
Let $\XX$ be a pair of graphs,
and consider the class notation of Definition~\ref{dc}
and the collections described in Definition~\ref{d3}.
For each of the following choices of a class,
if $\XX$ does not belong to the corresponding collection,
then the class contains infinitely many $\XX$-free graphs 
which are not $\omega$-colourable.
\begin{enumerate}
\item
For $\GG_{c, o, \alpha}$, the collection is $\OO_4^+$.
\item
For $\GG_{c, o}$, it is $\OO_{3}^+$.
\item
For $\GG_{c, \alpha}$, it is $\OO_{2c}^+$.
\item
For $\GG_{o, \alpha}$, it is $\OO_{2}^+$.
\item
For each of $\GG_{c}$, $\GG_{o}$, it is $\OO_{1}^+$.
\end{enumerate}
\end{prop}
The proof of Proposition~\ref{notOmega} is given below.
In addition, we show the following two observations.
\begin{obs}
\label{4K1D}
Each of the following conditions
is satisfied by infinitely many non-perfect graphs of independence $3$ which are distinct from an odd cycle.
\begin{enumerate}
\item
The graphs are $\{ 4K_1, Z_1 \}$-free.
\item
The graphs are connected and $\{ 4K_1, D\}$-free.
\item
The graphs are connected and $\{ 2K_1 \cup K_2, \overline{ K_1 \cup P_4} \}$-free.
\end{enumerate}
\end{obs}

\begin{proof}
We consider the graphs belonging to $\FF_{1}, \FF_{2}$ and $\FF_{3}$
(depicted in Figure~\ref{figFamilies123}),
and we note that the statement is satisfied subject to condition (1), (2) and (3), respectively.
\end{proof}

\begin{figure}[h!]
    \centering
    \includegraphics[scale=0.7]{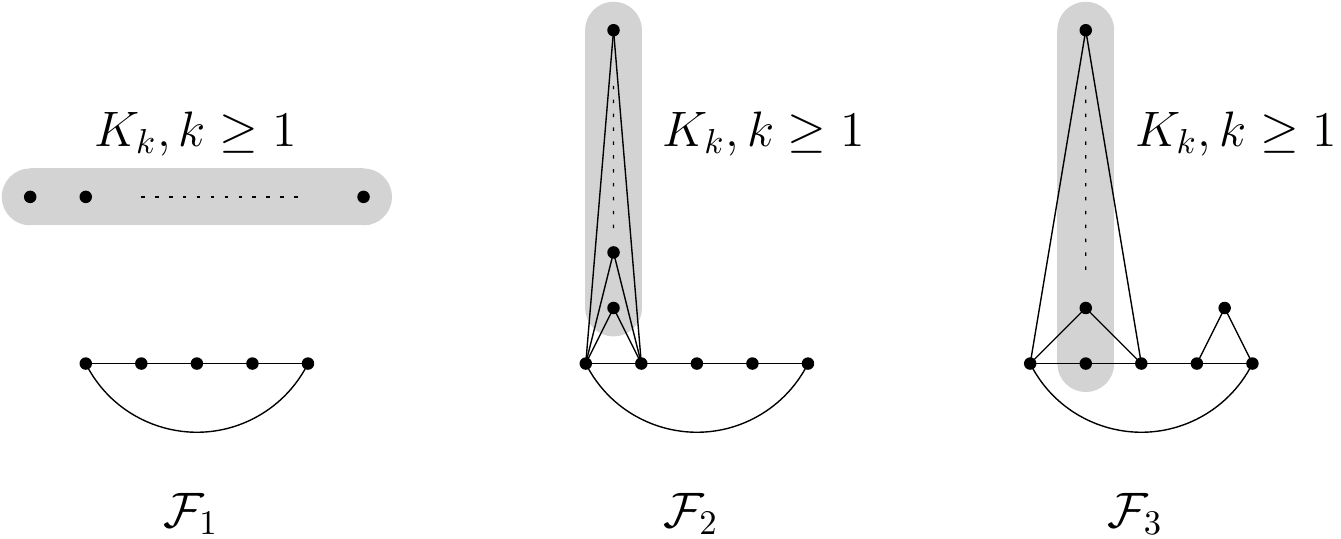}
    \caption{Families $\FF_1, \FF_2$ and $\FF_3$ of graphs.
	The grey ovals depict complete subgraphs $K_k$,
	and every choice of a positive integer $k$ gives a graph belonging to the family.
    The graphs of $\FF_1, \FF_2$ and $\FF_3$ are non-perfect and of independence~$3$.
    }
    \label{figFamilies123}
\end{figure}

\begin{obs}
\label{anException}
Each of the following conditions
is satisfied by some connected graph which is not $\omega$-colourable
and distinct from an odd cycle.
\begin{enumerate}
\item
The graph is $\{ 3K_1, K_4 \}$-free.
\item
The graph is $\{ 4K_1, K_3 \}$-free and of independence $3$.
\end{enumerate}
\end{obs}

\begin{proof}
For instance,
we consider graphs $F_6$ and $F_{10}$ depicted in Figure~\ref{figAnException},
and we note that they have the desired properties and
$F_6$ satisfies condition (1) and $F_{10}$ satisfies (2).
\end{proof}

We remark that there are precisely $13$, $14$ graphs satisfying
Observation~\ref{anException} subject to condition (1), (2), respectively.

\begin{figure}[h!]
    \centering
    \includegraphics[scale=0.7]{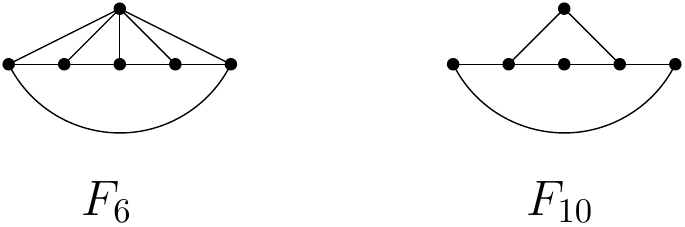}
    \caption{Graphs $F_6$ and $F_{10}$.
    (In fact, they are the smallest members of families
    $\FF_6$ and $\FF_{10}$, respectively.)
    }
    \label{figAnException}
\end{figure}

In the remainder of the section, we show Proposition~\ref{notOmega}.

\begin{proof}[Proof of Proposition~\ref{notOmega}]
\setcounter{claim}{0}
\setcounter{claimprefix}{\getrefnumber{notOmega}}
We consider families $\FF_{4}, \dots, \FF_{13}$ of graphs
(depicted in Figures~\ref{figFamilies4to10} and~\ref{figFamilies11to13}),
and we observe that no graph of $\FF_4, \dots, \FF_{13}$ is $\omega$-colourable. 
We let $X$ and $Y$ denote the graphs of $\XX$,
and we show statements (1) \dots (5).

\begin{figure}[t]
    \centering
    \includegraphics[scale=0.7]{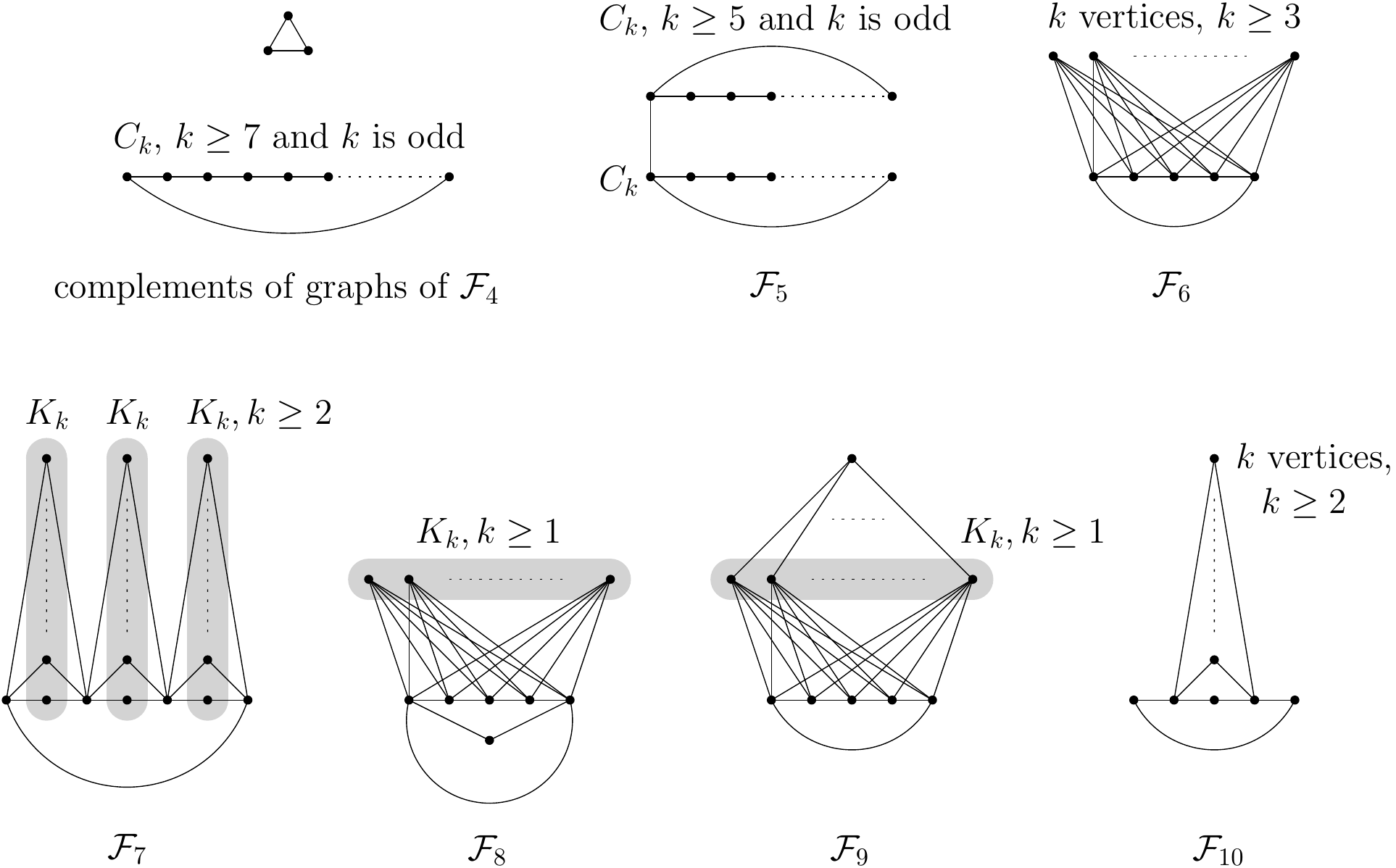}
    \caption{Family $\FF_4$ of graphs (given by complements)
    and families $\FF_5, \dots, \FF_{10}$.
	Similarly to Figure~\ref{figFamilies123},
	each grey oval depicts a complete subgraph.    
    The graphs of $\FF_4, \dots, \FF_{10}$ are connected
    and of independence at least $3$
    and not $\omega$-colourable.
    }
    \label{figFamilies4to10}
\end{figure}

We show statement (1), that is,
if $\XX$ does not belong to $\OO_4^+$, then
$\GG_{c, o, \alpha}$ contains infinitely many $\XX$-free graphs 
which are not $\omega$-colourable.
For the sake of clarity and efficiency of the proof, we show several claims.

\begin{claim} 
\label{c1}
No member of $\XX$ is $3K_1$ or an induced subgraph of $P_4$,
and we can assume that no member is $K_{1,3}$.
\end{claim}
\begin{proofcl}[Proof of Claim~\ref{c1}]
The first part of the claim follows from the assumption
that $\XX$ does not belong to $\OO_4^+$.
For the case that $K_{1,3}$ belongs to $\XX$,
we recall that the other member of $\XX$ is not an induced subgraph of $P_5$ or of $Z_2$
(since $\XX$ does not belong to $\OO_4^+$),
and so the statement of the proposition follows
by Theorem~\ref{alpha3}.
Thus, we can assume that $K_{1,3}$ does not belong to $\XX$.
\end{proofcl}

\begin{claim} 
\label{c2}
We can assume that $X$ is a forest and $Y$ is not.
Furthermore, we can assume that $X$ is $K_{1,4}$-free.
\end{claim}
\begin{proofcl}[Proof of Claim~\ref{c2}]
We suppose that either both members of $\XX$ are forests or none of them is,
and we show that the statement of the proposition is satisfied.
For the case that both are forests,
we note that each member of $\XX$ has at least four vertices
and is distinct from $P_4$ and $K_{1,3}$ (by Claim~\ref{c1}).
Hence, each member of $\XX$
contains some of $4K_1$, $2K_1 \cup K_2$, $K_1 \cup P_3$, $2K_2$
as an induced subgraph
(by item (1) of Lemma~\ref{forest}).
We consider family $\FF_4$ and note that it consists of
$\{4K_1, 2K_1 \cup K_2, K_1 \cup P_3, 2K_2\}$-free graphs.
Thus, every graph of $\FF_4$ is $\XX$-free,
and the statement is satisfied. 
In the latter case, we have that both members of $\XX$ contain a cycle,
and we consider family $\FF_5$
and observe that it contains infinitely many $\XX$-free graphs.
Consequently, we can assume that $X$ is a forest and $Y$ is not.
The `furthermore part' of the claim follows similarly by
considering~$\FF_5$.
\end{proofcl}

\begin{claim} 
\label{c3}
In addition, we can assume that $X$ is either $K_1 \cup P_3$ or contains some of
$4K_1, 2K_1 \cup K_2, 2K_2$ as an induced subgraph,
and that $Y$ is $\{ 4K_1, 2K_1 \cup K_2, K_1 \cup P_3, 2K_2, K_1 \cup K_3, C_5 \}$-free.
\end{claim}
\begin{proofcl}[Proof of Claim~\ref{c3}]
The first part of the claim follows by Claims~\ref{c1} and~\ref{c2}
and item (1) of Lemma~\ref{forest}.
For the case that $Y$ contains some of
$4K_1, 2K_1 \cup K_2, K_1 \cup P_3, 2K_2, K_1 \cup K_3, C_5$
as an induced subgraph,
we consider family $\FF_4$ and observe that the statement is satisfied.
\end{proofcl}

\begin{claim} 
\label{c4}
In addition, we can assume that
$X$ is distinct from $K_1 \cup \claw$ and $\claw^+$
and that $Y$ is $K_4$-free.
\end{claim}
\begin{proofcl}[Proof of Claim~\ref{c4}]
We suppose that $X$ is either $K_1 \cup \claw$ or $\claw^+$.
We recall that $Y$ is not a forest (by Claim~\ref{c2})
and that $Y$ is $K_1 \cup K_3$-free (by Claim~\ref{c3})
and distinct from $K_3$ and from $Z_1$
(since $X$ is $K_1 \cup \claw$ or $\claw^+$ and $\XX$ does not belong to $\OO_4^+$).
Thus, $Y$ is not an induced subgraph of $P_5$ or of $Z_2$,
and the statement follows by Theorem~\ref{alpha3}.

Next, we suppose that $Y$ contains $K_4$ as a subgraph.
If $X$ contains some of
$2K_1 \cup K_2, K_1 \cup P_3, 2K_2$
as an induced subgraph, then we consider family $\FF_6$
and note that the statement is satisfied.
Hence, we can assume that $X$ is $\{ 2K_1 \cup K_2, K_1 \cup P_3, 2K_2 \}$-free,
and so it contains induced $4K_1$ (by Claim~\ref{c3}).
Furthermore, $X$ is a forest and it is $K_{1,4}$-free (by Claim~\ref{c2}).
Consequently, we get that $X$ contains no edge.
In other words, $X$ is $kK_1$ (for some $k \geq 4$).
Thus, $Y$ has at least two non-edges (since $\XX$ does not belong to $\OO_4^+$).
Furthermore, we recall that $Y$ is $K_1 \cup K_3$-free (by Claim~\ref{c3}).
We consider the graphs of $\FF_7$ and observe that they are $\XX$-free.
\end{proofcl}

\begin{claim} 
\label{c5}
In addition, we can assume that $X$ is distinct from $K_1 \cup P_3$
and that $Y$ is $C_4$-free and that $Y$ contains $K_3$ as a subgraph.
\end{claim}
\begin{proofcl}[Proof of Claim~\ref{c5}]
We suppose that $X$ is $K_1 \cup P_3$.
In particular $Y$ is not an induced subgraph of $\overline{K_3 \cup P_4}$
(since $\XX$ does not belong to $\OO_4^+$).
We recall that $Y$ is 
$\{
4K_1,\allowbreak
2K_1 \cup K_2,\allowbreak
K_1 \cup P_3,\allowbreak
2K_2,\allowbreak
K_1 \cup K_3,\allowbreak
K_4,\allowbreak
C_5
\}$-free (by Claims~\ref{c3} and~\ref{c4}).
Consequently, $Y$ contains some of
$\overline{P_5}, K_{3,3}, K_{2,2,2}$ as an induced subgraph
(by Observation~\ref{cK3P4}),
and we consider family $\FF_6$.
Thus, we can assume that $X$ is distinct from $K_1 \cup P_3$.
In particular, $X$ contains some of
$4K_1, 2K_1 \cup K_2, 2K_2$ as an induced subgraph
(by Claim~\ref{c3}),
and we use this for showing the second part of the claim.

We suppose that $Y$ contains induced $C_4$, and we note that the statement is satisfied
by discussing two cases.
For the case that $X$ contains induced $2K_1 \cup K_2$, we consider family $\FF_8$.
Otherwise, we consider family $\FF_9$.

Hence, we can assume that $Y$ is $C_4$-free.
We recall that $Y$ is $\{ 2K_2, C_5 \}$-free (by Claim~\ref{c3}), 
and so it is, in fact, $\{ C_4, C_5, C_6, \dots \}$-free.
Furthermore, $Y$ is not a forest (by Claim~\ref{c2}),
and thus it contains $K_3$.
\end{proofcl}

\begin{claim} 
\label{c6}
In addition, we can assume that $X$ is
$(k + 2)K_1$ or $k K_1 \cup K_2$ (for some $k \geq 2$)
and that $Y$ is $3K_1$-free.
\end{claim}
\begin{proofcl}[Proof of Claim~\ref{c6}]
We recall that $Y$ contains $K_3$ (by Claim~\ref{c5}).
If $X$ contains some of
$2K_1 \cup P_3, K_1 \cup P_4, 2K_2$ as an induced subgraph,
then we consider family $\FF_{10}$.
Hence, we can assume that $X$ is 
$\{ 2K_1 \cup P_3, K_1 \cup P_4, 2K_2 \}$-free.
Furthermore, we recall that $X$ is $K_{1,4}$-free
(by Claim~\ref{c2})
and that $X$ is distinct from
$K_1 \cup P_3$, $K_1 \cup \claw$ and $\claw^+$
(by Claims~\ref{c4} and~\ref{c5}).
Consequently, we get that $X$ is
$(k + 2)K_1$ or $k K_1 \cup K_2$ for some $k \geq 2$
(by item (2) of Lemma~\ref{forest}).

We suppose that $Y$ contains induced $3K_1$.
We note that $Y$ contains at least one additional vertex (by Claim~\ref{c1}),
and hence it contains induced $\claw$ 
(since $Y$ is $\{ 4K_1, 2K_1 \cup K_2, K_1 \cup P_3\}$-free by Claim~\ref{c3}).
We conclude that the statement follows by Theorem~\ref{alpha3}.
\end{proofcl}

\begin{claim} 
\label{c7}
Consequently, we get that $Y$ is $\overline{K_1 \cup P_4}$.
\end{claim}
\begin{proofcl}[Proof of Claim~\ref{c7}]
We recall that $Y$ is $\{ 3K_1, 2K_2, C_4, C_5\}$-free
(by Claims~\ref{c3}, \ref{c5} and~\ref{c6}),
and we consider the complement of $Y$.
In particular, we get that $\overline{Y}$ is a forest.
Furthermore, $\overline{Y}$ is $\{ 4K_1, 2K_2, \claw\}$-free
(since $Y$ is $\{ K_1 \cup K_3, K_4, C_4 \}$-free
by Claims~\ref{c3}, \ref{c4} and~\ref{c5}).
Also, $\overline{Y}$ is distinct from $P_4$
(since $Y$ is distinct from $P_4$ by Claim~\ref{c1})
and distinct from $3K_1$ and $K_1 \cup P_3$
(since $Y$ is distinct from $K_3$ and $Z_1$ by the assumptions on $\XX$).
We note that we can apply item (2) of Lemma~\ref{forest}
and conclude that $\overline{Y}$ is $K_1 \cup P_4$.
\end{proofcl}

In particular, $X$ is distinct from $2K_1 \cup K_2$
(since $\XX$ does not belong to $\OO_4^+$).
Thus, $X$ contains induced $4K_1$ (by Claim~\ref{c6}).
The desired statement follows by considering family $\FF_7$.

\begin{figure}[t]
    \centering
    \includegraphics[scale=0.7]{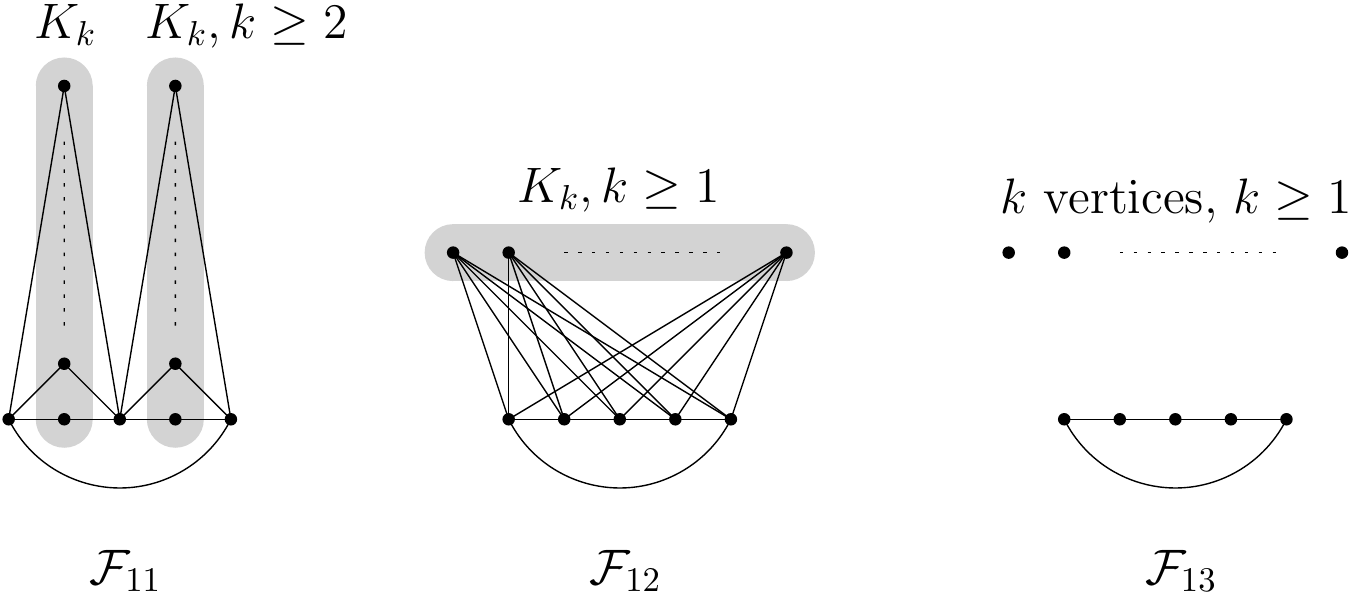}
    \caption{Families $\FF_{11}, \FF_{12}$ and $\FF_{13}$ of graphs.
	(The grey ovals depict complete subgraphs.)    
    The graphs are not $\omega$-colourable.
    }
    \label{figFamilies11to13}
\end{figure}

Next, we show (2).
We note that if $\XX$ does not belong to $\OO_4^+$,
then the statement follows from~(1).
Hence, we can assume that $\XX$ belongs to $\OO_4^+$ but not to $\OO_3^+$,
that is, we assume that at least one of the following is satisfied.
\begin{itemize}
\item
$X$ is $2K_1 \cup K_2$ and $Y$ is $\overline{ K_1 \cup P_4}$.
\item
$X$ is $\claw$ and $Y$ is either $K_1 \cup P_3$ or $2K_2$ or $P_5$ or $K_1 \cup K_3$ or $Z_2$. 
\item
$X$ is $K_1 \cup P_3$ and $Y$ is an induced subgraph of
$\overline{ K_3 \cup P_4}$
such that $Y$ is not an induced subgraph of $P_4$
and is distinct from $K_3$ and $Z_1$ and $D$.
\item
$X$ is $3K_1$ and $Y$ is not an induced subgraph of $P_4$
and is distinct from $K_k$ (for every $k \geq 3$) and from
$\overline{ \ell K_1 \cup K_2}$ (for every $\ell \geq 2$)
and from $Z_1$.
\end{itemize}
We note that $X$ contains $3K_1$ as an induced subgraph,
and we discuss $Y$ in more detail.
We observe that if the statement of the third item is satisfied,
then $Y$ contains some of $C_4, \overline{ K_1 \cup P_4}, \claw$
as an induced subgraph (by Observation~\ref{cK3P4on5}).
Furthermore, we show that if the statement of the last item is satisfied,
then $Y$ contains some of
$3K_1,\allowbreak
2K_2,\allowbreak
K_1 \cup K_3,\allowbreak 
C_4,\allowbreak 
C_5,\allowbreak
\overline{ K_1 \cup P_4},\allowbreak
\overline{ 2K_1 \cup P_3}$
as an induced subgraph.
For the sake of a contradiction, we suppose that $Y$ is
$\{
3K_1,\allowbreak 
2K_2,\allowbreak
K_1 \cup K_3,\allowbreak 
C_4,\allowbreak 
C_5,\allowbreak
\overline{ K_1 \cup P_4},\allowbreak
\overline{ 2K_1 \cup P_3}
\}$-free.
In other words,
$\overline{Y}$ is
$\{
K_3,\allowbreak 
C_4,\allowbreak 
\claw,\allowbreak 
2K_2,\allowbreak 
C_5,\allowbreak
K_1 \cup P_4,\allowbreak
2K_1 \cup P_3
\}$-free,
and so it is a
$\{
\claw,\allowbreak
2K_2,\allowbreak
K_1 \cup P_4,\allowbreak
2K_1 \cup P_3
\}$-free forest.
Furthermore, $\overline{Y}$ is 
not an induced subgraph of $P_4$
and it is distinct from
$k K_1$ (for every $k \geq 3$) and from
$ \ell K_1 \cup K_2$ (for every $\ell \geq 2$)
and from $K_1 \cup P_3$.
A contradiction follows by item (2) of Lemma~\ref{forest}.

We recall that $X$ contains induced $3K_1$,
and we discuss three cases for~$Y$.
For the case that $Y$ contains induced
$\overline{ K_1 \cup P_4}$ or $\overline{ 2K_1 \cup P_3}$,
we consider family~$\FF_{11}$ and observe that the statement is satisfied.
If $Y$ contains $C_5$ as an induced subgraph, then we consider the
family of all graphs whose complement is an odd cycle of length at least~$7$.
Otherwise, we conclude that $Y$ contains some of 
$3K_1, 2K_2, K_1 \cup K_3, C_4$ as an induced subgraph
and consider family~$\FF_{12}$.

We show (3).
Similarly as above,
we can assume that $\XX$ belongs to $\OO_4^+$ but not to $\OO_{2c}^+$.
Hence, $X$ contains $\claw$ as an induced subgraph
and $Y$ contains $K_3$,
and the desired statement follows by
considering the family of all odd cycles of length at least~$7$.

We show (4).
We can assume that $\XX$ belongs to $\OO_4^+$ but not to $\OO_{2}^+$,
and hence $X$ contains induced $\claw$
and $Y$ contains induced $2K_2$ or $K_3$,
and we consider family~$\FF_{13}$.

Lastly, we show the two statements given by (5).
We note that for $\GG_{c}$,
the statement follows from (2) and (3)
(since if $\XX$ does not belong to $\OO_{1}^+$,
then it does not belong to $\OO_{2c}^+$ or $\OO_3^+$).
Similarly, for $\GG_{o}$ the statement follows from (2) and (4).
\end{proof}

\section{Proving the main results}
\label{sProving}
Finally, we put together the statements shown in
Sections~\ref{sPerf}, \ref{sOmega} and~\ref{sFam}
and Theorem~\ref{alpha3},
and we prove
Theorems~\ref{mainNoExceptions} and~\ref{mainFiniteExceptions}.

\begin{proof}[Proof of Theorem~\ref{mainNoExceptions}]
For the sake of efficiency,
we shall first show the `if part' subject to items (1), (2), \dots, (10),
and then the `only if part' subject to items (10), (9), \dots, (1).

For the `if part' of item (1), we need to show that 
if $\XX$ belongs to $\PP_1$,
then every $\XX$-free graph of $\GG_5$ is perfect.
We recall that every $P_4$-free graph is perfect
(for instance, by Theorem~\ref{tSPGT}).
Hence, we can consider a graph of
$\GG_5$ and assume that it is 
$\{ 2K_1 \cup K_2, Z_1 \}$-free or
$\{ K_1 \cup P_3, Z_1 \}$-free or
$\{ K_1 \cup P_3, D \}$-free,
and the perfectness follows by
item (8), (9), (10) of Proposition~\ref{perf}, respectively.

For the `if part' of item (2), we need to show that 
if $\XX$ belongs to $\PP_2$,
then every $\XX$-free graph of $\GG_{\alpha}$ is perfect.
We note that if $\XX$ belongs to $\PP_1$,
then the statement follows by 
the `if part' of (1) shown above.
Also, if $\XX$ contains $3K_1$, then the considered restricted class is empty
and the statement is satisfied trivially.
Hence, we can assume that $\XX$
belongs to $\PP_2$ but neither to $\PP_1$ nor to $\II$.
In particular, one member of $\XX$ is $K_1 \cup P_3$
and the other is an induced subgraph of $\overline{K_3 \cup P_4}$,
and the statement follows by item (2) of Proposition~\ref{perf}.

We show the `if part' of (3).
Similarly as above,
we can assume that $\XX$ belongs to $\PP_{2c}$ but not to $\PP_2$.
Hence,
$\XX$ is either $\{ \claw, 2K_2 \}$ or $\{ \claw, P_5 \}$.
We note that every connected $\XX$-free graph of independence at least $3$
is, in fact, distinct from an odd cycle,
and the perfectness follows by Theorem~\ref{alpha3}.

For the `if part' of (4),
we can assume that $\XX$ belongs to $\PP_3$ but not to $\PP_1$.
Thus, $\XX$ is
either $\{ \claw^+, K_3 \}$ or $\{ \claw^+, Z_1 \}$
or $\{ \claw, K_3 \}$ or $\{ \claw, Z_1 \}$,
and the perfectness follows by item (3) of Proposition~\ref{perf}.

Similarly considering the `if part' of (5),
we can assume that $\XX$ is
either $\{ \claw, K_1 \cup K_3 \}$ or $\{ \claw, Z_2 \}$,
and the statement follows by Theorem~\ref{alpha3}.

Next, we show the `if part' of items (6), \dots, (10).
We note that for most pairs $\XX$,
the $\omega$-colourability follows by the  
`if part' of items (1), \dots, (5).
In~particular considering (6),
we can assume that $\XX$ is $\{ 2K_1 \cup K_2, D \}$.
We note that the statement follows by item (1) of Proposition~\ref{perf}
(since each of the graphs $E_1, \dots, E_4$ depicted in Figure~\ref{figE} is $\omega$-colourable).
The `if part' of (9) follows by the same argument.
Consequently for (7), (8) and (10),
we can assume that $\XX$ is $\{ 2K_1 \cup K_2, \overline{K_1 \cup P_4} \}$.
We conclude that the statement follows by Proposition~\ref{p1}.

In order to show the `only if part' of (10),
we need to show that if
$\XX$ does not belong to $\OO_4$,
then $\GG_{c, o, \alpha}$ contains an
$\XX$-free graph which are not $\omega$-colourable.
We note that for every pair $\XX$ not belonging to $\OO_4^+$,
the statement follows by item (1) of Proposition~\ref{notOmega}.
Hence,
we can assume that $\XX$ is 
either $\{ kK_1, K_{\ell} \}$
(for some $k \geq 4$ and $\ell \geq 3$),
or $\{ kK_1 \cup K_2, K_3 \}$
or $\{ (k+1)K_1, Z_1 \}$
or $\{ kK_1 \cup K_2, Z_1 \}$
or $\{ (k+1) K_1, D \}$
or $\{ k K_1 \cup K_2, D \}$
(for some $k \geq 3$),
or $\{ kK_1, \overline{ \ell K_1 \cup K_2 } \}$
(for some $k \geq 4$ and $\ell \geq 3$),
or $\{ K_1 \cup \claw, K_3 \}$
or $\{ K_1 \cup \claw, Z_1 \}$.
In all cases, the statement follows 
by item (2) of Observation~\ref{anException}.

Consequently for the `only if part' of (9),
we can assume that 
$\XX$ is none of the aforementioned pairs, and so $\XX$ belongs to $\OO_4$.
In addition, we can assume that $\XX$ belongs to $\OO_3^+$
(by item (2) of Proposition~\ref{notOmega}).
Hence, 
$\XX$ is either $\{ 3K_1, K_{k+1} \}$
or $\{ 3K_1, \overline{ kK_1 \cup K_2 } \}$
(for some $k \geq 3$)
and the statement follows by item (1) of Observation~\ref{anException}.

Similarly for (8), (7) and (6),
the statement follows by using item (3), (4) and (5) of Proposition~\ref{notOmega},
respectively.

Regarding the `only if part' of (5), \dots, (1),
we note that for most pairs $\XX$,
the existence of a desired non-perfect graph follows by  
the `only if part' of items (10), \dots, (6).
In particular, we can assume that $\XX$ is either
$\{ 2K_1 \cup K_2, D \}$ or
$\{ 2K_1 \cup K_2, \overline{K_1 \cup P_4} \}$.
We consider graphs $E_1, \dots, E_4$ depicted in Figure~\ref{figE},
and we note that they have the desired properties.
\end{proof}

We show Theorem~\ref{mainFiniteExceptions}.

\begin{proof}[Proof of Theorem~\ref{mainFiniteExceptions}]
We show the `if part' of the statement subject to items (1), (2), \dots, (11),
and then the `only if part' subject to items (11), (10), \dots, (1).

In order to show the `if part' of (1),
we need to show that 
if $\XX$ belongs to $\PP_1^+$,
then the set of all non-perfect $\XX$-free graphs is finite.
We recall that if $\XX$ belongs to $\PP_1$,
then all $\XX$-free graphs (except for $C_5$) are perfect
by item (1) of Theorem~\ref{mainNoExceptions}.
Also, if $\XX$ belongs to $\RR$
or $\XX$ is $\{ 3K_1, K_k \}$ (for some $k \geq 4$),
then there are only finitely many $\XX$-free graphs
by Theorem~\ref{Ramsey}, and so the statement is satisfied trivially.
Hence, we can assume that $\XX$ is
either $\{ 2K_1 \cup K_2, D \}$
or $\{ kK_1 \cup K_2, K_3 \}$
or $\{ 3K_1, \overline{ kK_1 \cup K_2 } \}$
(for some $k \geq 3$),
and the statement follows by items (1), (6) and (7) of Proposition~\ref{perf},
respectively.

For the `if part' of (2),
we can assume that $\XX$ belongs to $\PP_{1c}^+$ but not to $\PP_1^+$.
Hence, $\XX$ is
either $\{ (k+1)K_1, Z_1 \}$
or $\{ kK_1 \cup K_2, Z_1 \}$
(for some $k \geq 3$),
and the statement follows by item (5) of Proposition~\ref{perf}.

For the `if part' of (3) and (4),
we observe that the statement follows similarly as for (1) and (2).

Consequently for (5) and (6),
we can assume that $\XX$ is
either $\{ K_1 \cup \claw, K_3 \}$ or $\{ K_1 \cup \claw, Z_1 \}$.
The statement follows by item (4) of Proposition~\ref{perf}.

For the `if part' of (7),
we need to show that if $\XX$ belongs to $\OO_1^+$,
then the set of all $\XX$-free graphs which are not $\omega$-colourable is finite.
We can assume that $\XX$ does not belong to $\PP_1^+$,
and hence $\XX$ is
either $\{ (k+1) K_1, Z_1 \}$ 
or $\{ k K_1 \cup K_2, Z_1 \}$
or $\{ (k+1) K_1, D \}$ 
or $\{ k K_1 \cup K_2, D \}$
(for some $k \geq 3$),
or
$\{ k K_1, \overline{ \ell K_1 \cup K_2 } \}$
(for some $k \geq 4$ and $\ell \geq 3$),
and the statement follows by Proposition~\ref{omega}.

For (8), \dots, (11),
we note that the statement follows similarly
(in addition, using item (7) of Theorem~\ref{mainNoExceptions}
for the pair 
$\{ 2K_1 \cup K_2, \overline{K_1 \cup P_4} \}$).

Regarding the `only if part' of (11), \dots, (7),
we note that the statement follows by item (1), \dots, (5) of Proposition~\ref{notOmega},
respectively.

We show the `only if part' of (6).
We can assume that $\XX$ belongs to $\OO_4^+$
(otherwise, the statement is satisfied by the `only if part' of (11) discussed above).
Hence, we can assume that $\XX$ is 
either $\{ (k+1) K_1, D \}$
or $\{ k K_1 \cup K_2, D \}$
(for some $k \geq 3$),
or
$\{ k K_1, \overline{ \ell K_1 \cup K_2 } \}$
(for some $k \geq 4$ and $\ell \geq 3$),
or $\{ 2K_1 \cup K_2, \overline{ K_1 \cup P_4 } \}$,
and we note that the statement follows by
items (2) and (3) of Observation~\ref{4K1D}.

We conclude that
for items (5), \dots, (1), the statement follows similarly
(using items (1), (2) and (3) of Observation~\ref{4K1D}).
\end{proof}

\section*{Acknowledgements}
We thank the anonymous referee for their helpful comments.
The work of the first author was partially supported by DMS-EPSRC grant DMS-2120644.
The work of the second and fourth author was supported by projects 17-04611S and 20-09525S of the Czech Science Foundation.
The second author was also supported by the MUNI Award in Science and Humanities of the Grant Agency of Masaryk University.
The work of the third author was supported by the National
Natural Science Foundation of China (Nos.~12171393, 12071370)
and the Natural Science Basic Research Program of Shaanxi (Nos.~2021JM-040, 2020JQ-099).

\small{
}

\end{document}